\documentclass[10pt]{article}
\usepackage[margin=1in]{geometry}

\usepackage{graphicx}
\usepackage[T1]{fontenc}
\usepackage[utf8]{inputenc}
\usepackage[english]{babel}
\usepackage{amsmath,amsthm,amssymb,amsfonts}
\usepackage{lmodern}
\usepackage{subcaption}
\usepackage{booktabs}
\usepackage{tikz}

\usepackage[backend=biber,style=alphabetic,citestyle=alphabetic,url=true,sorting=nyt,maxbibnames=99,giveninits=true]{biblatex}
\AtBeginBibliography{\small}

\usepackage{csquotes}
\usepackage{makecell}

 \usepackage{empheq}
 \definecolor{darkgreen}{rgb}{0,0.5,0}
 \definecolor{darkblue}{rgb}{0,0.08,0.45}
 \definecolor{someblue}{rgb}{0,0.08,0.9}
\definecolor{rust}{rgb}{0.6,0.1,0.1}

\usepackage{hyperref}
\hypersetup{
  colorlinks=true,
  linkcolor=someblue,
  citecolor=rust,
  urlcolor=darkgreen
}

\usepackage{algorithm,algorithmic}

\usepackage[framemethod=TikZ]{mdframed}
\mdfsetup{roundcorner=8pt}

\usepackage{textcomp}
\usepackage{bbm}

\usepackage{enumerate}
\usepackage{cleveref}

\newtheorem{theorem}{Theorem}[section]
\newtheorem{corollary}[theorem]{Corollary}
\newtheorem{lemma}[theorem]{Lemma}

\theoremstyle{definition}
\newtheorem{remark}[theorem]{Remark}
\newtheorem{definition}{Definition}

\definecolor{lavendergray}{rgb}{0.77, 0.76, 0.82}
\definecolor{lightgray}{rgb}{0.83, 0.83, 0.83}
\definecolor{mygray}{gray}{0.9}
\newmdtheoremenv[backgroundcolor=mygray,innertopmargin=-2pt]{assumption}{Assumption}

\newcommand{\n}[1]{\|#1 \|}
\DeclareMathOperator{\E}{\mathbb{E}}

\renewcommand{\a}{\alpha}
\renewcommand{\b}{\beta}
\renewcommand{\d}{\delta}

\renewcommand{\phi}{\varphi}
\renewcommand{\t}{\tau}
\newcommand{\s}{\sigma}
\newcommand{\e}{\varepsilon}
\renewcommand{\th}{\theta}

\newcommand{\Gap}{\mathrm{Gap}}

\newcommand{\x}{\bar x}
\newcommand{\y}{\bar y}
\newcommand{\z}{\bar z}
\newcommand{\w}{\bar w}

\newcommand{\bx}{\mathbf{x}}
\newcommand{\bw}{\mathbf{w}}
\newcommand{\by}{\mathbf{y}}
\newcommand{\bz}{\mathbf{z}}
\newcommand{\bu}{\mathbf{u}}
\newcommand{\bv}{\mathbf{v}}

\newcommand{\bzb}{\mathbf{\bar z}}
\newcommand{\bwb}{\mathbf{\bar w}}
\newcommand{\bub}{\mathbf{\bar u}}
\newcommand{\bvb}{\mathbf{\bar v}}

\newcommand{\R}{\mathbb R}

\newcommand{\N}{\mathbb N}
\newcommand{\Z}{\mathbb Z}

\newcommand{\cF}{\mathcal F}

\newcommand{\cO}{\mathcal O}
\newcommand{\cX}{\mathcal X}
\newcommand{\cY}{\mathcal Y}

\newcommand{\cC}{\mathcal C}

\newcommand{\cZ}{\mathcal Z}

\newcommand{\Fr}{\mathrm{Frob}}

\renewcommand{\empty}{\varnothing}
\newcommand{\tr}{\top}

\newcommand{\dgf}{h} %
\newcommand{\grad}{\nabla \dgf} %
\newcommand{\D}{D}
\newcommand{\DD}[2]{\dgf(#1)-\dgf(#2)-\lr{\grad(#2), #1-#2}}

\newcommand{\Fhat}{\widehat F}

\newcommand{\one}{\mathbbm{1}}

\newcommand{\Ex}[2]{\mathbb{E}_{#2} \left[ #1 \right]}
\newcommand{\lr}[1]{\langle #1\rangle}
\newcommand{\Sol}{\mathsf{Sol}}
\DeclareMathOperator{\prox}{prox}
\DeclareMathOperator*{\argmin}{argmin}
\DeclareMathOperator{\dom}{dom}

\DeclareMathOperator{\nnz}{nnz}
\DeclareMathOperator{\rank}{rank}
\DeclareMathOperator{\sign}{sign}

\usepackage[colorinlistoftodos,prependcaption,backgroundcolor=black!5!white,bordercolor=red]{todonotes}

\title{Stochastic Variance Reduction for Variational Inequality Methods}
\author{Ahmet Alacaoglu\footnote{University of Wisconsin-Madison, USA,
    email: \href{mailto:alacaoglu@wisc.edu}{alacaoglu@wisc.edu}}
  \qquad \qquad Yura Malitsky\footnote{Link\"oping University, Sweden, email: \href{mailto:yurii.malitskyi@liu.se}{yurii.malitskyi@liu.se}
 }}
\date{}

\begin{document}

\maketitle

\begin{abstract}
  We propose stochastic variance reduced algorithms for solving convex-concave saddle point problems, monotone variational inequalities, and monotone inclusions. Our framework applies to extragradient, forward-backward-forward, and forward-reflected-backward methods both in Euclidean and Bregman setups. All proposed methods converge in the same setting as their deterministic counterparts and they either match or improve the best-known complexities for solving structured min-max problems. Our results reinforce the correspondence between variance reduction in variational inequalities and minimization. We also illustrate the improvements of our approach with numerical evaluations on matrix games.
\end{abstract}

\section{Introduction}
In this paper, we focus on solving variational inequalities (VI):
\begin{equation}\label{eq: prob_vi}
\text{find~~~} \bz_\ast \in \mathcal{Z}\text{~~~such that~~~}\langle F(\bz_\ast), \bz - \bz_\ast  \rangle + g(\bz) - g(\bz_\ast) \geq 0,\quad \forall \bz\in\mathcal{Z},
\end{equation}
where $F$ is a monotone operator and $g$ is a proper convex
lower semicontinuous function.
This formulation captures optimality conditions for minimization/saddle
  point problems, see~\cite[Section 1.4.1]{facchinei2007finite}.

In the last decade there have been at least two surges of interest to
VIs. Both were motivated by the need to solve min-max problems. The first surge
came from the realization that many nonsmooth problems can be solved more
efficiently if they are formulated as saddle point problems~\cite{nesterov2005smooth,nemirovski2004prox,chambolle2011first,esser2010general}.
The second has been started by machine learning community, where
solving nonconvex-nonconcave saddle point problems became of paramount
importance~\cite{gidel2018variational,gemp2018global,mertikopoulos2018optimistic}.
Additionally, VIs  have
    applications in game theory, control theory, and differential equations, see~\cite{facchinei2007finite}.

    A common structure encountered in min-max problems is that the
    operator $F$ can be written as a finite-sum:
    $F = F_1+\dots + F_N$, see~\Cref{sec: apps} for concrete examples.  Variance reduction techniques use this
    specific form to improve the complexity of deterministic
    methods in minimization. Existing results on variance reduction
    for saddle point problems show that these techniques improve the
    complexity for bilinear problems compared to deterministic
    methods.  However, in general these methods require stronger
    assumptions to converge than the latter do (see~\Cref{table:
      t1}). At the same time, stochastic methods that have been shown to converge
    under only monotonicity do not have complexity advantages over the
    deterministic methods.

Such a dichotomy does not exist in minimization: variance reduction comes with no extra assumptions.  This points out to a
fundamental lack of understanding for its use in saddle point
problems. Our work shows that there is indeed a natural correspondence
between variance reduction in variational inequalities and
minimization.
In particular, we propose stochastic variants of
extragradient (EG), forward-backward-forward (FBF), and
forward-reflected-backward (FoRB) methods which converge under mere
monotonicity. For the bilinear case our results match the best-known
complexities, while for the nonbilinear, we do not require bounded domains
as in the previous work and we improve the best-known complexity by a
logarithmic factor, using simpler algorithms.
Recently,~\cite{han2021lower} established the optimality of our algorithms with matching lower bounds, for solving (potentially nonbilinear) convex-concave min-max problems with finite sum form.

We also show application of our techniques for solving monotone
inclusions and strongly monotone problems. Our results for monotone
inclusions potentially improve the rate of deterministic methods
(depending on the Lipschitz constants) and they seem to be the first
such result in the literature.  We illustrate practical benefits
of our new algorithms by comparing with deterministic methods and an
existing  variance reduction scheme in~\Cref{sec: numexp}.

\subsection{Related works}\label{sec: rel_work}
\paragraph{Variational inequalities.}
The standard choices for solving VIs have been methods such as
extragradient (EG)/Mirror-Prox
(MP)~\cite{korpelevich1976extragradient,nemirovski2004prox},
forward-backward-forward (FBF)~\cite{tseng2000modified}, dual extrapolation~\cite{nesterov2007dual} or
reflected gradient/forward-reflected-backward
(FoRB)~\cite{malitsky2015projected,malitsky2018forward}\footnote{In the unconstrained
  setting, this method is also known as Optimistic Mirror Descent
  (OMD) or Optimistic Gradient Descent Ascent
  (OGDA)~\cite{rakhlin2013online,daskalakis2018training} and is also equivalent to the classical Popov's method~\cite{popov1980modification}}.
  These
methods differ in the number of operator calls and projections (or
proximal operators) used each iteration, and consequently, can be
preferable to one another in different settings.  The standard
convergence results for these algorithms include global iterates'
convergence, complexity $\cO(\e^{-1})$ for monotone
problems and linear rate of convergence for strongly monotone
problems.
\begin{table}
\captionsetup{font=small}
\centering
\begin{tabular}{lll}\toprule
     & \textbf{Assumptions} & \textbf{Complexity} \\\midrule
EG/MP, FBF, FoRB$^\dagger$ & $F$ is monotone & $\cO\left(\frac{  N L_F}{\varepsilon} \right)$\\[1mm]
\midrule
 EG/MP$^\ddagger$ & \makecell[l]{$F$ is monotone \& $\bz \mapsto \langle
  F(\bz)+\tilde{\nabla} g(\bz), \bz - \bu \rangle$ \\is convex  for any $ \bu$} & $\cO\left( N+\frac{\sqrt{N} L}{\varepsilon}
                     \right)$ \\[1mm]
\midrule
 EG/MP$^\ddagger$ & $F$ is monotone \&
  bounded domains & $\tilde{\cO}\left( N+\frac{\sqrt{N} L}{\varepsilon}
                     \right)$ \\[1mm]
\midrule
FoRB$^*$ & $F$ is monotone & $\cO\left( N+\frac{N L}{\varepsilon}
\right)$\\[1mm]
\midrule
\makecell[l]{\textbf{This paper}\\EG/MP, FBF, FoRB} & $F$ is monotone & $\cO\left(N+\frac{\sqrt{N} L}{\varepsilon}
                     \right)$\\
\bottomrule
\end{tabular}
\caption{Table of algorithms with $F(\bz) = \sum_{i=1}^N F_{i}(\bz)$. EG: Extragradient, MP: Mirror-Prox,
FBF: forward-backward-forward, FoRB:
forward-reflected-backward.
$\tilde{\nabla} g$ denotes a subgradient of $g$.
$^\dagger$\cite{korpelevich1976extragradient,tseng2000modified,nemirovski2004prox,malitsky2018forward}, $^\ddagger$\cite{carmon2019variance}, $^*$\cite{alacaoglu2020simple}.
}
\label{table: t1}
\end{table}
\paragraph{Variance reduction.}
Variance reduction has revolutionized stochastic methods in optimization. This technique applies to finite sum minimization problem of the form
$\min_{\bx} \frac 1 N\sum_{i=1}^N f_i(\bx)$.
Instead of using a random sample $\mathbf{g}_k = \nabla f_i(\bx_k)$ as SGD does, variance reduction methods use
\begin{equation}
\mathbf{g}_k =\nabla f(\bw_k) + \nabla f_i(\bx_k)-\nabla
  f_i(\bw_k).\label{eq: vr_gk}
  \end{equation}
  A good choice of $\bw_k$ decreases the ``variance''
$\E \n{\mathbf{g_k}-\nabla f(\bx_k)}^2$ compared to
$\E\n{\nabla f_i(\bx_k)-\nabla f(\bx_k)}^2$ that SGD has.  A simple idea that is easy
to explain to undergraduates, easy to implement, and most importantly
that provably brings us a better convergence rate than pure
SGD and GD in a wide range of scenarios. Classical works include
\cite{johnson2013accelerating,defazio2014saga}. For a more thorough
list of references, see the recent review~\cite{gower2020variance}.
\paragraph{Variance reduction and VIs.}
One does not need to be meticulous to quickly find finite sum
problems where existing variance reduction methods do not work. In the
convex world, the first that comes to mind is non-smoothness. As already mentioned, saddle point reformulations often come to
rescue.

The work \cite{balamurugan2016stochastic} was seminal in using
variance reduction for  saddle point problems and monotone inclusions in general.
In particular, the authors studied stochastic variance reduced
variants of forward-backward algorithm and proved linear convergence under strong monotonicity.
For bilinearly coupled problems, the complexity in~\cite{balamurugan2016stochastic} improves the deterministic method in the strongly monotone setting.
\cite{chavdarova2019reducing} developed an extragradient method with variance reduction and analyzed its convergence under strong monotonicity assumption.
Unfortunately, the worst-case complexity in this work was less favorable than~\cite{balamurugan2016stochastic}.

Strong monotonicity may seem like a fine assumption, similar to strong
convexity in minimization.  While algorithmically it is true,
in applications with min-max, the former is far less
frequent. For instance, the operator $F$ associated with a convex-concave saddle point problem is monotone,
but not strongly monotone without further assumptions. Thus, it is crucial to remove this assumption.

An influential work in this direction is by~\cite{carmon2019variance}, where the authors proposed a randomized
variant of Mirror-Prox. The authors focused primarily on matrix games
and for this important case, they improved complexity over
deterministic methods. However, because of this specialization, more
general cases required additional
assumptions. In particular, for problems beyond matrix games, the
authors assumed that either
$\bz \mapsto \langle F(\bz) + \tilde{\nabla} g(\bz), \bz-\bu \rangle$
is convex for all $\bu$~\cite[Corollary 1]{carmon2019variance} or that
domain is bounded~\cite[Algorithm 5, Corollary 2]{carmon2019variance}: in particular, domain diameter is used as a parameter for this algorithm.
As one can check, the former might not hold even for convex
minimization problems with $F=\nabla f$. The latter, on the other
hand, while already restrictive, requires a more complicated
three-loop algorithm, which incurred an additional logarithmic factor
into total complexity.

There are other works that did not improve complexity but
introduced new ideas.  An algorithm similar in spirit to ours is due
to~\cite{alacaoglu2020simple}, where variance reduction is applied to
FoRB. This algorithm was the first to converge under only
monotonicity, but did not improve complexity of deterministic
methods.  Several works studied VI methods in the stochastic
setting and showed slower rates with decreasing step
sizes~\cite{mishchenko2020revisiting,bohm2020two}, or increasing
mini-batch
sizes~\cite{iusem2017extragradient,bot2019forward,cui2019analysis}, or extra assumptions~\cite{gorbunov2022stochastic}.

\subsection{Outline of results and comparisons}\label{sec: outline_results}
Throughout the paper, we assume access to a stochastic oracle $F_\xi$ such that $\mathbb{E}[F_\xi(\bz)] = F(\bz)$.
\paragraph{Complexity and $\varepsilon$-accurate solution.} A point $\bar{\bz}$ is an $\varepsilon$-accurate solution if $\mathbb{E}\left[\Gap(\bar{\bz})\right] \leq \varepsilon$, where $\Gap$ is defined in~\Cref{sec: euc_sub_rate}.
Complexity of the algorithm is defined as the number of calls to
$F_\xi$ to reach an $\varepsilon$-accurate solution. In general,  we suppose that evaluation of $F$ is $N$ times more expensive than
$F_\xi$. For specific problems with bilinear coupling, we measure the complexity in terms of arithmetic operations.
\paragraph{Nonbilinear finite-sum problems.} We consider the problem~\eqref{eq: prob_vi}
with $F = \sum_{i=1}^N F_i$ where $F$ is monotone, $L_F$-Lipschitz, and it is $L$-Lipschitz in mean, in view of~\Cref{asmp: asmp1}(iv).
In this setting, our variance reduced variants of
EG, FBF, and FoRB (\Cref{cor:eucl},~\Cref{cor:
  fbf},~\Cref{cor: forb}) have complexity
$\mathcal{O}\big( N + \sqrt{N}L\varepsilon^{-1}
\big)$ compared to the deterministic methods with
$ \mathcal{O}\left( NL_F\varepsilon^{-1}
\right)$.

Our methods improve over deterministic variants as long as $L \leq
\sqrt{N}L_F$.
This is a similar improvement over deterministic complexity,
as accelerated variance reduction does for minimization problems~\cite{woodworth2016tight,allen2017katyusha}.
To our knowledge, the only precedent with a  result similar to ours is the work
~\cite{carmon2019variance}, where spurious assumptions were
required (see~\Cref{sec: rel_work} and~\Cref{table: t1}), complexity had additional logarithmic terms and a complicated three-loop algorithm was needed.
\paragraph{Bilinear problems.} When we focus on bilinear problems (App.~\ref{sec: bilinear}), the complexity of our methods is
$\tilde{\mathcal{O}}\left( \nnz(A)+ \sqrt{\nnz(A)(m+n)}L \varepsilon^{-1} \right)$,
where $L = \| A\|_\Fr$ with Euclidean setup and $L=\|A\|_{\max}$ with simplex constraints and the entropic setup.
In contrast, the complexity of deterministic method is
$\tilde{\mathcal{O}}\left( \nnz(A) L_F\varepsilon^{-1} \right)$,
where $L_F = \| A\|$ with Euclidean setup and $L_F = \| A\|_{\max}$ with the entropic setup.
Our complexity shows strict improvements over deterministic methods
when $A$ is dense.
Our variance reduced variants for FBF and FoRB enjoy similar guarantees and obtain the same complexities (\Cref{cor: fbf},~\Cref{cor: forb}).

In both settings this complexity was first obtained by~\cite{carmon2019variance}.
Our results generalize the set of problems where this complexity applies due to less assumptions (for example, linearly constrained convex optimization) and also use more practical/simpler algorithms (see~\Cref{sec: numexp} for an empirical comparison).
Note that our variance reduced Mirror-Prox in Alg.~\ref{alg: mp1} is different from the Mirror-Prox variant in~\cite[Alg. 1, Alg. 2]{carmon2019variance}.
\paragraph{How to read the paper?}
We summarize the main results in~\Cref{table: t2}.
We recommend a reader, who wants a quick grasp of
the idea, to refer to~\Cref{sec: euclidean}. This should be sufficient for
understanding our main technique.
The extension to
Bregman case is technical in nature and noticing the reason for using a double loop algorithm in this case requires a good deal of understanding of proposed analysis.

For the most
general case with Bregman distances, a reader can skip~\Cref{sec: euclidean} without losing much and go directly to~\Cref{sec:br}.
We kept~\Cref{sec: euclidean} for a clearer exposition of the main ideas via a simpler algorithm and analysis.
We tried to make the sections self-contained and the proofs isolated: convergence
rate and convergence of iterates are separated.

Finally, one can read~\Cref{sec: extensions} right after~\Cref{sec: euclidean} to see how the same ideas give rise to variance reduced FBF and FoRB algorithms with similar guarantees and ability to solve monotone inclusions.
In this section, we also illustrate how to obtain linear rate of convergence with strong convexity.
\Cref{sec: apps} clarifies how to apply our developments to specific problems such as matrix games and linearly constrained optimization.
Most of the proofs are given with the corresponding results; remaining proofs are deferred to~\Cref{sec: appendix}.
\paragraph{Practical guide. } 
We give the parameters recommended in practice in~\Cref{rem: eg},~\Cref{rem: mp},~\Cref{rem: forb} for~\Cref{alg: eg1},~\Cref{alg: mp1},~\Cref{alg: forb}, respectively.
These parameters are optimized to obtain the best complexity in terms of dependence to problem dimensions (and not dependence to constants) and we use them in our numerical experiments in~\Cref{sec: numexp}.
For convenience we also specify the updates in the important case of
matrix games with entropic setup in~\Cref{sec: matrix_games}.
\begin{table}
\captionsetup{font=small}
\centering
\begin{tabular}{lll}\toprule
     & \textbf{Rate \& Complexity} & \textbf{Convergence of iterates} \\\midrule
\makecell[l]{\textbf{Euclidean setup}\\
EG: \Cref{sec: euclidean} \\
FBF, FoRB: \Cref{sec: extensions}
} & \makecell[l]{~\\\Cref{eq: eg1_th2}, \Cref{cor:eucl}\\ \Cref{cor: fbf},~\Cref{cor: forb}} & \makecell[l]{~\\\Cref{th: eg1_th1}\\ \Cref{th: fbf},~\Cref{th: forb}} \\[1mm]
\midrule
\makecell[l]{\textbf{Bregman setup}\\
MP: \Cref{sec:br}} & \makecell{~\\\Cref{eq: eg1_th_br}, \Cref{cor:compl_br}}
                                               & \qquad\qquad\quad--- \\[1mm]
\bottomrule
\end{tabular}
\caption{Structure of the paper}
\label{table: t2}
\end{table}

\section{Euclidean setup}\label{sec: euclidean}
To illustrate our technique, we pick extragradient method due to the simplicity of its analysis, its extension to Bregman distances and its wide use in the literature.
\subsection{Preliminaries}\label{sec: prelim}
Let $\mathcal{Z}$ be a finite dimensional vector space with Euclidean
inner product $\lr{\cdot,\cdot}$ and norm $\n{\cdot}$.
The notation $[N]$ represents the set $\{1, \dots, N\}$.
We say $F\colon \dom{g}\to \mathcal{Z}$ is monotone if for all $\bx, \by$, $\langle F(\bx) - F(\by), \bx - \by \rangle \geq 0$.
Proximal operator is defined as $\prox_g(\bx) = \argmin_\by
\left\{g(\by) + \frac{1}{2} \| \by - \bx \|^2\right\}$. For  a proper convex lower semicontinuous~(lsc) $g$, domain is defined as $\dom g = \{\bz\colon g(\bz) < +\infty\}$ and the following prox-inequality is standard
\begin{equation}\label{eq: prox_ineq}
\mathbf{\bar z} = \prox_{g}(\bz) \quad \iff\quad \langle \mathbf{\bar z} - \bz,
\bu - \mathbf{\bar z} \rangle \geq  g(\bzb) - g(\bu), \qquad \forall \bu\in\mathcal{Z}.
\end{equation}
We continue with our assumptions.
\begin{assumption}\label{asmp: asmp1}~
\begin{enumerate}[(i)]
\item The solution set $\mathsf{Sol}$ of~\eqref{eq: prob_vi} is
  nonempty.
  \item The function $g\colon\mathcal{Z}\to\mathbb{R}\cup\{+\infty\}$ is proper convex lower semicontinuous.
\item The operator $F$ is monotone.
\item The operator $F$ has a stochastic oracle $F_\xi$ that
  is unbiased $F(\bz)=\mathop{\mathbb{E}}\left[ F_\xi(\bz) \right]$ and $L$-Lipschitz in mean:
\begin{equation*}
{\mathbb{E}}\left[\| F_\xi(\bu) - F_\xi(\bv) \|^2\right] \leq L^2 \| \bu-\bv \|^2,~~~~~\forall \bu, \bv\in\mathcal{Z}.
\end{equation*}
\end{enumerate}
\end{assumption}
\paragraph{Finite sum.} Suppose $F$ has a finite sum representation
$F = \sum_{i=1}^N F_i$, where each $F_i$ is $L_i$-Lipschitz and the
full operator $F$ is $L_F$-Lipschitz. By triangle inequality it
follows, of course, that $L_F\leq \sum_{i=1}^N L_i$. On one hand,
$\sum_{i=1}^N L_i$ can be much larger than $L_F$. On the other, it
might be the case that $L_i$ are easy to compute, but not a true
$L_F$. Then the latter inequality gives us the most natural upper bound
on $L_F$. The two simplest stochastic oracles can be defined as follows
\begin{enumerate}
\item Uniform sampling: $ F_\xi(\bz) = NF_i(\bz),~ q_i=\Pr\{\xi = i\} = \frac 1 N$.   In this case, $L = \sqrt{N\sum_{i\in[N]} L_i^2}$.
\item Importance sampling: $F_\xi(\bz) = \frac{1}{q_i}F_i(\bz),\quad q_i=\Pr\{\xi = i\} =  \frac{L_i}{\sum_{j\in [N]}L_j}$.
  In this case, $L = \sum_{i\in [N]}L_i$.
\end{enumerate}

This example is useful in several regards. First, it is one of the
most general problems that proposed algorithms can tackle and for
concreteness it is useful to keep it as a reference point.  Second,
this problem even in its generality already indicates possible
pitfalls caused by non-optimal stochastic oracles. If  $L$ of our
stochastic oracle is much worse (meaning larger) than $L_F$, it may
eliminate all advantages of cheap stochastic oracles.
In the sequel, for finite-sum problems, we assume that $\xi \in [N]$, similar to the two oracles described above.

\subsection{Extragradient with variance reduction}
The classical stochastic variance reduced gradient
(SVRG)~\cite{johnson2013accelerating} uses a double loop structure
(looped): the full gradients are computed in the outer loop and the
cheap variance reduced gradients~\eqref{eq: vr_gk} are used in the
inner loop. Works~\cite{kovalev2020dont,hofmann2015variance} proposed
a \emph{loopless} variant of SVRG, where the outer loop was eliminated and
instead full gradients were computed \emph{once in a while} according to a
randomized rule. Both methods share similar guarantees, but the
latter variant is slightly simpler to analyze and implement.

We present the loopless version of extragradient with variance reduction in Alg.~\ref{alg:
  eg1}. Every iteration requires two stochastic oracles $F_\xi$ and
one $F$ with probability $p$.
Parameter $\a$ is the key in establishing a favorable
complexity. While convergence of $(\bz_k)$ to a solution will be proven for
any $\a\in [0,1)$, a good total complexity requires a specific
choice of $\a$.
Therefore, the specific form of $\bzb_k$ is important.
Later, we see that with $\a = 1-p$, Alg.~\ref{alg: eg1}
has the claimed complexity in~\Cref{table: t1}.
\begin{algorithm}[t]
\caption{Extragradient with variance reduction }
\begin{algorithmic}
    \STATE {\bfseries Input:} Set $p\in (0,1]$, probability
    distribution $Q$, step size
    $\tau$, $\a \in (0,1)$, $\mathbf{z}_0=\mathbf{w}_0$
    \vspace{.2cm}
    \FOR{$k = 0,1,\ldots $}
        \STATE $\mathbf{\z}_k = \a \mathbf{z}_k + (1-\a)\mathbf{w}_k$
        \STATE $\mathbf{z}_{k+1/2} = \prox_{\tau g}(\mathbf{\z}_k - \tau F(\mathbf{w}_k))$
        \STATE Draw an index $\xi_k$ according to $Q$
        \STATE $\mathbf{z}_{k+1} = \prox_{\tau g}(\mathbf{\z}_k - \tau[F(\mathbf{w}_k)+F_{\xi_k}(\mathbf{z}_{k+1/2}) - F_{\xi_k}(\mathbf{w}_k)])$
        \STATE $\mathbf{w}_{k+1} = \begin{cases}
          \mathbf{z}_{k+1}, \text{~~with probability~~} p\\
                \mathbf{w}_k, \text{~~~~with probability~~} 1-p
            \end{cases}$
        \ENDFOR
      \end{algorithmic}
\label{alg: eg1}
\end{algorithm}
It is interesting to note that by eliminating all randomness,
Alg.~\ref{alg: eg1} reduces to extragradient.
\begin{remark}\label{rem: eg}
For running Alg.~\ref{alg: eg1} in practice, we suggest $p=\frac{2}{N}$, $\a = 1 - p$, and $\t = \frac{0.99\sqrt p}{L}$. Specific problem may require
a more careful examination of ``optimal'' parameters (see App.~\ref{sec: bilinear}).
\end{remark}
\subsection{Analysis}\label{subs:eucl}
In Alg.~\ref{alg: eg1}, we have two sources of randomness at each iteration:
the index $\xi_k$ which is used for computing $\bz_{k+1}$ and the choice of $\bw_{k}$ (the snapshot point).
We use the following notation for the conditional expectations: $\mathbb{E}[\cdot | \sigma(\xi_0, \dots, \xi_{k-1},\bw_k)] = \mathbb{E}_k[\cdot]$ and $\mathbb{E}[\cdot | \sigma(\xi_0, \dots, \xi_{k},\bw_k)] = \mathbb{E}_{k+1/2}[\cdot]$.

For the iterates $(\bz_k)$, $(\bw_k)$ of Alg.~\ref{alg: eg1} and any $\bz\in
\dom g$,  we define
\begin{align*}
\Phi_{k}(\bz) &\coloneqq \a \| \bz_{k} - \bz \|^2 + \frac{1-\a}{p}\n{\bw_{k}-\bz}^2.
\end{align*}
We see in the following
lemma how $\Phi_{k}$ naturally arises in our analysis as the
Lyapunov function.
  \begin{lemma}\label[lemma]{lem: eg1}
Let~\Cref{asmp: asmp1} hold, $\alpha \in [0, 1)$, $p\in(0, 1]$, and
$\tau =  \frac{\sqrt{1-\a}}{L}\gamma$, for $\gamma \in (0, 1)$. Then for $(\bz_k)$ generated by Alg.~\ref{alg: eg1} and any $\bz_\ast\in\Sol$, it holds that
\begin{equation*}
\mathbb{E}_k \left[ \Phi_{k+1}(\bz_\ast) \right] \leq \Phi_k(\bz_\ast) - (1-\gamma) \Big((1-\alpha) \| \bz_{k+1/2} - \bw_{k} \|^2 + \mathbb{E}_k\| \bz_{k+1} - \bz_{k+1/2} \|^2\Big).
\end{equation*}
    Moreover, it holds that $\sum_{k=0}^\infty \Big((1-\a)\mathbb{E}
    \|\bz_{k+1/2} - \bw_k \|^2  + \mathbb{E}
    \|\bz_{k+1}-\bz_{k+1/2} \|^2  \Big) \leq \frac 1
{1-\gamma}\Phi_0(\bz_\ast)$.
\end{lemma}
\begin{proof}
A reader may find it simpler to follow the analysis by assuming that $g$ is the indicator function of some convex set. Then since all iterates are feasible, we would have $g(\bz_k) = 0$.

Let us denote $\Fhat(\bz_{k+1/2}) = F(\bw_k) + F_{\xi_k}(\bz_{k+1/2}) - F_{\xi_k}(\bw_k)$.
By prox-inequality~\eqref{eq: prox_ineq} applied to the definitions of $\bz_{k+1}$ and $\bz_{k+1/2}$, we have that for all $\bz$,
\begin{align}
\begin{split}\label{eq: prox_ineqs}
\langle \bz_{k+1} - \bzb_k + \tau \Fhat(\bz_{k+1/2}), \bz - \bz_{k+1} \rangle &\geq \tau g(\bz_{k+1}) -\tau g(\bz), \\
\langle \bz_{k+1/2} - \bzb_k + \tau F(\bw_k), \bz_{k+1} - \bz_{k+1/2} \rangle &\geq \tau g(\bz_{k+1/2}) - \tau g(\bz_{k+1}).
\end{split}
\end{align}
We sum two inequalities, use the definition of $\Fhat(\bz_{k+1/2})$, and arrange to get
\begin{align}\label{eq: eg1_eq1}
\langle \bz_{k+1} - \bzb_k, \bz - \bz_{k+1} \rangle &+  \langle
\bz_{k+1/2} -  \bzb_k, \bz_{k+1} - \bz_{k+1/2} \rangle \notag \\
&+\tau \langle
F_{\xi_k}(\bw_k) - F_{\xi_k}(\bz_{k+1/2}), \bz_{k+1}-\bz_{k+1/2}
\rangle \notag \\ &
+\tau \langle \Fhat(\bz_{k+1/2}), \bz - \bz_{k+1/2} \rangle  \geq \tau[g(\bz_{k+1/2}) - g(\bz)].
\end{align}
For the first inner product we use definition of $\bzb_k$ and identity $2\langle \mathbf{a},\mathbf{b} \rangle = \|\mathbf{a}+\mathbf{b}\|^2 - \|\mathbf{a}\|^2 - \|\mathbf{b}\|^2$
\begin{align}
2&\langle \bz_{k+1} - \bzb_k, \bz-\bz_{k+1} \rangle = 2\alpha\langle \bz_{k+1} - \bz_k, \bz-\bz_{k+1} \rangle + 2(1-\alpha)\langle \bz_{k+1} - \bw_k, \bz-\bz_{k+1} \rangle \notag\\
&=\alpha\left( \| \bz_k - \bz\|^2 - \|\bz_{k+1} - \bz \|^2 - \|\bz_{k+1} - \bz_k\|^2\right) \notag \\
&\qquad + (1-\alpha)\big( \| \bw_k - \bz\|^2 -  \| \bz_{k+1} - \bz \|^2 - \| \bz_{k+1} - \bw_k\|^2 \big) \notag \\
&=\alpha \| \bz_k - \bz\|^2 - \| \bz_{k+1} - \bz \|^2 + (1-\alpha) \| \bw_k - \bz\|^2 - \alpha \| \bz_{k+1} - \bz_k \|^2 - (1-\alpha) \| \bz_{k+1}-\bw_k \|^2.\label{eq: eg1_ip1}
\end{align}
Similarly, for the second inner product in~\eqref{eq: eg1_eq1} we deduce
\begin{multline}
2\langle \bz_{k+1/2} - \bzb_k, \bz_{k+1} - \bz_{k+1/2} \rangle = \alpha \| \bz_{k+1} - \bz_{k} \|^2 - \| \bz_{k+1} - \bz_{k+1/2} \|^2 + (1-\alpha) \| \bz_{k+1} - \bw_k\|^2 \\
-\alpha \| \bz_{k+1/2} - \bz_k \|^2 - (1-\alpha) \| \bz_{k+1/2} - \bw_k \|^2.\label{eq: eg1_ip2}
\end{multline}
For the remaining terms in~\eqref{eq: eg1_eq1}, we plug in $\bz=\bz_\ast$, use that $\bz_{k+1/2}, \bw_k$ is deterministic under the conditioning of $\mathbb{E}_k$ and $\mathbb{E}_k[\Fhat(\bz_{k+1/2})]=\mathbb{E}_k\left[ F(\bw_k) + F_{\xi_k}(\bz_{k+1/2}) - F_{\xi_k}(\bw_k)\right] = F(\bz_{k+1/2})$ to obtain
  \begin{align}
  \mathbb{E}_k &\left[ \langle \Fhat(\bz_{k+1/2}), \bz_\ast-\bz_{k+1/2} \rangle +g(\bz_\ast) -
    g(\bz_{k+1/2})\right] \notag \\
  &=  \langle  F(\bz_{k+1/2}), \bz_\ast-\bz_{k+1/2} \rangle +g(\bz_\ast) -
  g(\bz_{k+1/2}) \notag \qquad \qquad \,\left(\E_k[F(\bw_k)-F_{\xi_k}(\bw_k)]=0\right)\\
&\leq \langle  F(\bz_\ast), \bz_\ast-\bz_{k+1/2} \rangle +g(\bz_\ast) -
g(\bz_{k+1/2}) \leq 0 \qquad\qquad\,\left(\text{monotonicity and \eqref{eq: prob_vi}}\right)
    \label{eq: eg1_ip2.5}
  \end{align}
and
\begin{align}
&    \mathbb{E}_k\left[2\tau \langle F_{\xi_k}(\bw_k) -
      F_{\xi_k}(\bz_{k+1/2}) , \bz_{k+1} - \bz_{k+1/2} \rangle\right] \notag\\
    \leq\, &\Ex{ 2\tau \n{F_{\xi_k}(\bw_k) - F_{\xi_k}(\bz_{k+1/2})}
      \n{\bz_{k+1} - \bz_{k+1/2} }}{k}    \qquad\qquad\qquad
    ~(\text{Cauchy-Schwarz})\notag \\
         \leq\, &\frac{\tau^2}{\gamma} \E_k\left[\|F_{\xi_k}(\bz_{k+1/2}) - F_{\xi_k}(\bw_k )\|^2\right] + \gamma
    \mathbb{E}_k\left[\| \bz_{k+1} - \bz_{k+1/2}\|^2\right]   \qquad
    ~(\text{Young's ineq.})\notag \\
    \leq\, &(1-\a)\gamma \|\bz_{k+1/2} - \bw_k \|^2 + \gamma
    \mathbb{E}_k\left[\| \bz_{k+1} - \bz_{k+1/2}\|^2\right]. \qquad
    \text{(\Cref{asmp: asmp1}(iv))} \label{eq:
      eg1_ip3}
\end{align}
We use~\eqref{eq: eg1_ip1},~\eqref{eq: eg1_ip2},~\eqref{eq: eg1_ip2.5}, and~\eqref{eq: eg1_ip3} in~\eqref{eq: eg1_eq1}, after taking expectation $\mathbb{E}_k$ and letting $\bz=\bz_\ast$, to deduce
\begin{align}
\mathbb{E}_k\left[\| \bz_{k+1} -\bz_\ast \|^2 \right] \leq \alpha
\|\bz_k-\bz_\ast\|^2 + (1-\alpha)\|\bw_k &-\bz_\ast\|^2 -\left(1-\alpha
\right)(1-\gamma) \|\bz_{k+1/2} -\bw_k\|^2 \notag \\ &
- \left( 1-\gamma \right) \mathbb{E}_k\left[\| \bz_{k+1} - \bz_{k+1/2} \|^2\right].\label{eq: eg1_main1}
\end{align}
By the definition of $\bw_{k+1}$ and $\mathbb{E}_{k+1/2}$, it follows that
\begin{equation}\label{eq: eg1_wexp}
\frac{1-\alpha}{p}\mathbb{E}_{k+1/2}\left[\| \bw_{k+1} -\bz_\ast \|^2\right] = (1-\alpha)\| \bz_{k+1}-\bz_\ast \|^2 + (1-\alpha)\left(\frac{1}{p}-1\right)  \| \bw_k-\bz_\ast \|^2.
\end{equation}
We add~\eqref{eq: eg1_wexp} to~\eqref{eq: eg1_main1} and apply the tower property $\mathbb{E}_k\left[ \mathbb{E}_{k+1/2}[\cdot] \right] =\mathbb{E}_k[\cdot]$ to deduce
\begin{multline*}
\alpha\mathbb{E}_k\left[ \| \bz_{k+1} - \bz_\ast \|^2 \right] + \frac{1-\alpha}{p} \mathbb{E}_k \left[ \| \bw_{k+1} - \bz_\ast \|^2 \right] \leq \alpha \| \bz_k - \bz_\ast \|^2 + \frac{1-\alpha}{p} \| \bw_k - \bz_\ast \|^2 \\
- (1-\gamma)\Big( (1-\alpha)\|\bz_{k+1/2} - \bw_k\|^2 + \mathbb{E}_k\left[\| \bz_{k+1} - \bz_{k+1/2}\|^2 \right]\Big).
\end{multline*}
Using the definition of $\Phi_k(\bz)$, we obtain the first
result. Applying total expectation and summing the inequality
yields the second result.
\end{proof}

To
show the almost sure convergence of the sequence $(\bz_k)$, we need $F_\xi$ to be continuous for all $\xi$. For a
finite sum example it follows automatically from \Cref{asmp: asmp1}.
The proof is given in~\Cref{sec: appendix}.
\begin{theorem}\label{th: eg1_th1}
Let~\Cref{asmp: asmp1} hold, $F_\xi$ be continuous for all $\xi$, $\alpha\in[0,1)$, $p\in(0,1]$,
 and $\tau =  \frac{\sqrt{1-\a}}{L}\gamma$, for $\gamma \in (0, 1)$. Then, almost surely there exists $\bz_\ast \in \mathsf{Sol}$ such that $(\bz_k)$ generated by Alg.~\ref{alg: eg1} converges to $\bz_\ast$.
\end{theorem}
\subsubsection{Convergence rate and complexity for monotone case}\label{sec: euc_sub_rate}
In the general monotone case, the convergence measure is the gap
function given by
\begin{equation*}
  \mathrm{Gap}(\bw) = \max_{\bz\in\mathcal{C}}\bigl\{ \langle F(\bz), \bw - \bz \rangle + g(\bw) - g(\bz)\bigr\},
\end{equation*}
where $\mathcal{C}$ is a compact subset of $\cZ$ that we use to handle
the possibility of unboundedness of $\dom g$ (see~\cite[Lemma 1]{nesterov2007dual}).
Since we work in probabilistic setting,
naturally our convergence measure will be based on $\E[\mathrm{Gap}(\bw)]$.
We start with a simple lemma for ``switching'' the order of maximum
and expectation, which is required for showing convergence of expected
gap. This technique is standard for such
purpose~\cite{nemirovski2009robust} and the proof is given in~\Cref{sec: appendix}.
\begin{lemma}\label[lemma]{lem: exp_max_lem}
Let $\mathcal{F}=(\mathcal{F}_k)_{k\geq 0}$ be a filtration and $(\bu_k)$ a stochastic process adapted to $\mathcal{F}$ with $\mathbb{E}[\bu_{k+1} | \mathcal{F}_k] = 0$. Then for any $K \in \mathbb{N}$, $ \bx_0\in \cZ$, and any compact set $\mathcal{C}\subset\mathcal{Z}$,
\begin{equation*}
\mathbb{E}\left[\max_{\bx\in\mathcal{C}} \sum_{k=0}^{K-1} \langle \bu_{k+1}, \bx \rangle\right] \leq \max_{\bx\in\mathcal{C}}\frac{1}{2}\|  \bx_0 - \bx\|^2 +\frac{1}{2} \sum_{k=0}^{K-1} \mathbb{E}\| \bu_{k+1} \|^2.
\end{equation*}
\end{lemma}
We now continue with the main result of this section.
\begin{theorem}\label{eq: eg1_th2}
  Let~\Cref{asmp: asmp1} hold,  $p\in(0, 1]$, $\a = 1-p$, and
$\tau =  \frac{\sqrt{1-\a}}{L}\gamma$, for $\gamma \in (0, 1)$. Then,
  for $\bz^{K} = \frac{1}{K}\sum_{k=0}^{K-1} \bz_{k+1/2}$, it follows that
\begin{equation*}
  \mathbb{E}\left[ \mathrm{Gap}(\bz^{K}) \right] = \cO \left(
    \frac{L}{\sqrt p K} \right).
\end{equation*}
In particular, for $\t = \frac{\sqrt p}{2L}$, the rate is
\(  \mathbb{E}\left[ \mathrm{Gap}(\bz^{K}) \right] \leq
  \frac{17.5L}{\sqrt p K}  \max_{z\in\mathcal{C}}\n{\bz_0-\bz}^2
\).
\end{theorem}
Recall that we measure complexity in terms of calls to the stochastic oracle $F_\xi(\cdot)$ and we assumed that the cost of computing $F(\cdot)$ is $N$ times  that of $F_\xi(\cdot)$. For a finite sum
example, this is a natural assumption.
\begin{remark}\label[remark]{rem: complexity_remark}
  For Alg.~\ref{alg: eg1}, since per iteration cost is $pN+2$ calls to $F_\xi$
  \textit{in expectation}, the result is ``average'' total complexity:
  \textit{expected number of calls to get a small expected gap}.
\end{remark}
\begin{corollary}\label[corollary]{cor:eucl}
  In the setting of~\Cref{eq: eg1_th2}, the average
  total complexity
   of Alg.~\ref{alg: eg1} to reach $\e$-accuracy is $\cO \left( N+ (pN+2) \left(1+\frac{L}{\sqrt p \e}\right) \right)$.
  In particular, for $p = \frac 2 N$ it is $\cO \left( N + \frac{\sqrt N L}{\e} \right)$.
\end{corollary}

\begin{proof}[Proof of~\Cref{eq: eg1_th2}]
  As we have already mentioned, when all randomness is eliminated, that is
  $F_\xi = F$ and $p=1$, \Cref{alg: eg1} reduces to the extragradient. In
  that case, the convergence rate $\cO(1/K)$ would follow almost
  immediately from the proof of~\Cref{lem: eg1}. In a stochastic
  setting the proof is more subtle and we have to rely on~\Cref{lem:
    exp_max_lem} to deal with the error terms caused by
  randomness. Let
\begin{equation*}
\Theta_{k+1/2}(\bz) = \langle F(\bz_{k+1/2}), \bz_{k+1/2} -\bz\rangle + g(\bz_{k+1/2}) - g(\bz).
\end{equation*}
We will proceed as in~\Cref{lem: eg1} before getting~\eqref{eq:
  eg1_main1}. In particular, using~\eqref{eq: eg1_ip1} and~\eqref{eq: eg1_ip2} in~\eqref{eq: eg1_eq1} gives
\begin{align}\label{eq:rate_first}
  2\tau\Theta_{k+1/2}(\bz) +  \|\bz_{k+1} - \bz \|^2
  & \leq \alpha \| \bz_k - \bz\|^2 + (1-\a)\| \bw_k - \bz \|^2 \notag\\
  & \quad + 2\tau \langle F_{\xi_k}(\bw_k) - F_{\xi_k}(\bz_{k+1/2}), \bz_{k+1} - \bz_{k+1/2} \rangle \notag\\
  & \quad - (1-\alpha) \|\bz_{k+1/2} -\bw_k\|^2 -\| \bz_{k+1} - \bz_{k+1/2} \|^2 \notag\\
  & \quad + \underbrace{2\t\lr{F(\bz_{k+1/2})-F_{\xi_k}(\bz_{k+1/2})-F(\bw_k)+F_{\xi_k}(\bw_k), \bz_{k+1/2}-\bz},}_{e_1(\bz,k)}
\end{align}
where we call the last term by $e_1(\bz,k)$.

Now, we set $\alpha=1-p$. We want to rewrite
\eqref{eq:rate_first} using
$\Phi_{k}(\bz) = (1-p) \n{\bz_{k}-\bz}^2 + \n{\bw_k-\bz}^2$. For this, we need to add $ \| \bw_{k+1} - \bz \|^2 - \| \bw_k - \bz \|^2$ to both sides. Then, we define the error
\begin{align}
  e_2(\bz, k) &= p \n{\bw_k-\bz}^2+
  \n{\bw_{k+1}-\bz}^2 - \n{\bw_k-\bz}^2  - p \n{\bz_{k+1}-\bz}^2 \notag \\
  &=2\langle p\bz_{k+1} + (1-p)\bw_k - \bw_{k+1}, \bz \rangle - p\|\bz_{k+1} \|^2 - (1-p)\|\bw_k \|^2 + \|\bw_{k+1} \|^2.\label{eq: wje4}
\end{align}
With this at hand, we can cast~\eqref{eq:rate_first}  as
\begin{align*}
  2\tau\Theta_{k+1/2}(\bz) + \Phi_{k+1}(\bz)
  & \leq \Phi_{k}(\bz) +e_1(\bz,k) + e_2(\bz,k) \\
  & \quad +2\tau \langle F_{\xi_k}(\bw_k) - F_{\xi_k}(\bz_{k+1/2}), \bz_{k+1} - \bz_{k+1/2} \rangle \\
  & \quad - p \|\bz_{k+1/2} -\bw_k\|^2 -\| \bz_{k+1} - \bz_{k+1/2} \|^2.
\end{align*}
We sum this inequality over $k=0,\dots,K-1$, take maximum of both sides over $\bz\in\cC$, and
then take total expectation to obtain
\begin{align}
  2\t K\mathbb{E}\left[ \mathrm{Gap}(\bz^K) \right]
  &\leq \max_{\bz\in\mathcal{C}}\Phi_0(\bz) + \mathbb{E}\left[\max_{\bz\in\mathcal{C}}\sum_{k=0}^{K-1} \big(e_1(\bz, k) +
    e_2(\bz, k) \big)\right] \notag \\
  &\quad - \mathbb{E}\sum_{k=0}^{K-1}\Big(\| \bz_{k+1} - \bz_{k+1/2} \|^2+p \|\bz_{k+1/2} -\bw_k\|^2\Big) \notag\\
  &\quad +2\tau \mathbb{E}\sum_{k=0}^{K-1}\left[\langle F_{\xi_k}(\bw_k) - F_{\xi_k}(\bz_{k+1/2}), \bz_{k+1} - \bz_{k+1/2} \rangle\right]\label{eq: one_before_bound_gap}
\end{align}
where we used $\mathbb{E}\left[ \max\limits_{\bz\in\mathcal{C}} \sum\limits_{k=0}^{K-1}
  \Theta_{k+1/2}(\bz) \right] \geq K\mathbb{E}\left[ \mathrm{Gap}(\bz^K)
\right]$, which follows from monotonicity of $F$, linearity of
$\bz_{k+1/2}\mapsto \langle F(\bz), \bz_{k+1/2} - \bz \rangle$, and convexity of $g$.

The tower property, the estimation from~\eqref{eq: eg1_ip3}, and $1-\alpha=p$
applied on~\eqref{eq: one_before_bound_gap} imply
\begin{equation}\label{eq:bound_gap}
2\t K \mathbb{E}\left[ \mathrm{Gap}(\bz^K) \right] \leq \max_{\bz\in\mathcal{C}}\Phi_0(\bz) + \mathbb{E}\left[\max_{\bz\in\mathcal{C}}\sum_{k=0}^{K-1} \big(e_1(\bz, k) + e_2(\bz, k)\big) \right].
\end{equation}
Therefore, the proof will be complete upon deriving an upper bound for the second term on RHS.
We will instantiate~\Cref{lem: exp_max_lem} twice for bounding this
term. First, for $e_1(\bz, k)$ we set in~\Cref{lem: exp_max_lem},
\begin{alignat*}{3}
&\mathcal{F}_k = \sigma(\xi_0, \dots, \xi_{k-1}, \bw_k), ~~~&&\tilde
\bx_0 = \bz_0, ~~~&&\bu_{k+1} = 2\tau \left([F_{\xi_k}(\bz_{k+1/2}) - F_{\xi_k}(\bw_k)] - [F(\bz_{k+1/2}) - F(\bw_k)] \right),
\end{alignat*}
where by definition we set $\cF_0 = \s(\xi_0, \xi_{-1}, \bw_0)= \s(\xi_0)$. With this, we obtain the bound
\begin{align}\label{eq:bound_for_e1}
  \mathbb{E} \left[\max_{\bz\in\mathcal{C}} \sum_{k=0}^{K-1} e_1(\bz, k)\right] &=
  \mathbb{E} \left[\max_{\bz\in\mathcal{C}} \sum_{k=0}^{K-1}\lr{\bu_{k+1},\bz}\right] -
  \mathbb{E}\left[\sum_{k=0}^{K-1}\lr{\bu_{k+1},\bz_{k+1/2}} \right]=  \mathbb{E}
 \left[ \max_{\bz\in\mathcal{C}} \sum_{k=0}^{K-1}\lr{\bu_{k+1},\bz}\right] \notag\\
  &\leq \max_{\bz\in \cC}\frac 12 \n{\bz_0-\bz}^2+
\frac{1}{2} \sum_{k=0}^{K-1}\E \n{\bu_{k+1}}^2 \notag\\
  &\leq \max_{z\in \cC}\frac12 \n{\bz_0-\bz}^2+ 2\tau^2 L^2 \sum_{k=0}^{K-1}\E \n{\bz_{k+1/2}-\bw_k}^2,
\end{align}
where the second equality follows by the tower property, $\mathbb{E}_k [\bu_{k+1}] = 0$, and $\mathcal{F}_k$-measurability of $\bz_{k+1/2}$.
The last inequality is due to
\[\E\n{\bu_{k+1}}^2 =\E\left[\E_k\n{\bu_{k+1}}^2\right]\leq4\tau^2\E\left[\E_k\n{F_{\xi_k}(\bz_{k+1/2}) -
    F_{\xi_k}(\bw_k)}^2\right]\leq 4\tau^2 L^2 \E\n{\bz_{k+1/2}-\bw_k}^2, \]
    where we use the tower property, $\E \n{X-\E X}^2 \leq
\E\n{X}^2$, and~\Cref{asmp: asmp1}(iv).

Secondly, we set in \Cref{lem: exp_max_lem}
\begin{alignat*}{3}
&\mathcal{F}_k = \sigma(\xi_0, \dots, \xi_{k}, \bw_k), \qquad &&\tilde \bx_0 =
\bz_0, \qquad &&\bu_{k+1} =  p\bz_{k+1} +(1-p)\bw_k - \bw_{k+1},
\end{alignat*}
and use  $\mathbb{E}\left[\E_{k+1/2}[ \|\bw_{k+1} \|^2- p\|\bz_{k+1} \|^2 - (1-p)\|\bw_k \|^2 ]\right]=0$, to obtain the bound
\begin{align}\label{eq:bound_for_e2}
  \mathbb{E} \left[\max_{\bz\in\mathcal{C}} \sum_{k=0}^{K-1} e_2(\bz, k)\right] &=
   2\E\left[\max_{\bz\in \cC}\sum_{k=0}^{K-1}\lr{\bu_{k+1},z}\right]\notag\\ &\leq
    \max_{\bz\in \cC}\n{\bz_0-\bz}^2 +
\sum_{k=0}^{K-1}    \E\n{\bu_{k+1}}^2\notag \\
  &=
    \max_{\bz\in \cC}\n{\bz_0-\bz}^2 +
     p(1-p)\sum_{k=0}^{K-1}\E\n{\bz_{k+1}-\bw_k}^2,
 \end{align}
where the  inequality follows from~\Cref{lem: exp_max_lem} and
the second equality from the derivation
\begin{align*}
  \E\n{\bu_{k+1}}^2
  &=\E\left[\E_{k+1/2}\n{\bu_{k+1}}^2\right]= \E\left[\E_{k+1/2}\n{\E_{k+1/2}\left[\bw_{k+1}\right]-\bw_{k+1}}^2\right] \\
  &= \E \left[\E_{k+1/2}\n{\bw_{k+1}}^2 -\n{\E_{k+1/2}\left[\bw_{k+1}\right]}^2\right] \\
  &= \E\left[p\n{\bz_{k+1}}^2+(1-p)\n{\bw_k}^2 -\n{p\bz_{k+1}+(1-p)\bw_k}^2\right] \\
  &= p(1-p)\E\n{\bz_{k+1}-\bw_k}^2,
\end{align*}
which uses $\E\n{X-\E X}^2= \E \n{X}^2
- \n{\E X}^2$.

Combining \eqref{eq:bound_for_e1}, \eqref{eq:bound_for_e2}, and
\eqref{eq:bound_gap}, we finally arrive at
\begin{align}\label{eq:for_corol}
  2\t K \mathbb{E}\left[ \mathrm{Gap}(\bz^K) \right] &\leq
  \max_{\bz\in\mathcal{C}}\Phi_0(\bz) +\max_{\bz\in \cC}\frac 12 \n{\bz_0-\bz}^2+ 2\tau^2
  L^2 \sum_{k=0}^{K-1}\E \n{\bz_{k+1/2}-\bw_k}^2  \notag \\
  &\quad + \max_{\bz\in \cC}\n{\bz_0-\bz}^2 + p(1-p)\sum_{k=0}^{K-1}\E\n{\bz_{k+1}-\bw_k}^2
\end{align}
We have to estimate terms under the sum:
\begin{align}
&\quad \E\left[ \sum_{k=0}^{K-1}\left(2\tau^2 L^2 \n{\bz_{k+1/2}-\bw_k}^2 +
     p(1-p)\n{\bz_{k+1}-\bw_k}^2  \right)\right]\notag\\
& \leq p \E\left[ \sum_{k=0}^{K-1}\left(2 \n{\bz_{k+1/2}-\bw_k}^2
    +\n{\bz_{k+1}-\bw_k}^2  \right)\right]\notag\\
&\leq p \E\left[ \sum_{k=0}^{K-1}\left((2+\sqrt 2) \n{\bz_{k+1/2}-\bw_k}^2
    +(2+\sqrt 2) \n{\bz_{k+1}-\bz_{k+1/2}}^2  \right)\right]\notag\\
& \leq \frac{2+\sqrt 2}{1-\gamma}\Phi_0(\bz_*) \leq \frac{3.5}{1-\gamma} \max_{\bz\in \cC}\Phi_0(\bz), \label{eq: ineq_sub_last}
\end{align}
where the first inequality in~\eqref{eq: ineq_sub_last} uses~\Cref{lem: eg1}
and $1-\a = p$.

Now we will use that $\bw_0=\bz_0$ and, hence, $\Phi_0(\bz) =
(2-p)\n{\bz_0-\bz}^2\leq 2 \n{\bz_0-\bz}^2$ in~\eqref{eq:for_corol}. This yields
\begin{align*}
    2\t K \mathbb{E}\left[ \mathrm{Gap}(\bz^K) \right] &\leq \left(2
      + \frac 32 + \frac{7}{1-\gamma}\right)\max_{\bz\in
      \cC}\n{\bz_0-\bz}^2 = 7\left(\frac 12 + \frac{1}{1-\gamma}\right)\max_{\bz\in \cC}\n{\bz_0-\bz}^2.
  \end{align*}
  Finally, using  $\t = \frac{\sqrt{p}\gamma}{L}$, we obtain
\[ \mathbb{E}\left[ \mathrm{Gap}(\bz^K) \right] \leq \frac{7L}{2\sqrt p
 \gamma K}\left(\frac 12 + \frac{1}{1-\gamma}\right)\max_{\bz\in
 \cC}\n{\bz_0-\bz}^2 = \cO \left( \frac{L}{\sqrt p K} \right).
\]
In particular, with a stepsize $\t = \frac{\sqrt p}{2L}$, the right-hand side reduces to $\frac{17.5L}{\sqrt p K}\max_{\bz\in
 \cC}\n{\bz_0-\bz}^2$.
\end{proof}

\begin{proof}[Proof of \Cref{cor:eucl}]
In average each iteration costs $pN + 2
$ calls to $F_\xi$. To reach $\e$-accuracy we need $\left\lceil\cO \left(
    \frac{L}{\sqrt p \e} \right)\right\rceil$ iterations.
Hence, the total average complexity is $\cO \left( \frac{(pN+2)L}{\sqrt p \e}\right)$.
Finally, the optimal choice  $p=\frac{2}{N}$ gives $\cO \left(
  \frac{\sqrt{N}L}{\e}\right)$ complexity.
\end{proof}

To see the justification for the choice of $\alpha=1-p$, consider the
proof with any choice of $\alpha$. The resulting bound will be
$\mathcal{O}\left( \frac{1}{\sqrt{1-\alpha}} +
  \frac{\sqrt{1-\alpha}}{p} \right)$. Then $\alpha=1-p$ optimizes it in terms of $p$ dependence.
  Our rate guarantee on Theorem~\ref{eq: eg1_th2} is on the averaged iterate $\bz^K$, which is shown to be necessary to get the $O(1/K)$ rate for deterministic extragradient in~\cite{golowich2020last}.

\section{Bregman setup}\label{sec:br}
\subsection{Preliminaries}
In this section, we assume that $\mathcal{Z}$ is a normed vector space with a dual
space $\mathcal{Z}^\ast$ and primal-dual norm pair $\|\cdot\|$ and
$\|\cdot \|_\ast$. Let $\dgf\colon \cZ \to \mathbb{R}\cup\{+\infty\}$
be a proper convex lsc function that satisfies \textit{(i)}
$\dom g\subseteq \dom h$, \textit{(ii)} $h$ is differentiable over
$\dom \partial h$, \textit{(iii)} $h$ is $1$-strongly convex on $\dom g$. Then
we can define the Bregman distance
$D\colon \dom g\times \dom \partial \dgf\to \R_+ $ associated with $h$
by
\[\D(\bu,\bv)\coloneqq \DD{\bu}{\bv}.\]
  Note that since $\dgf$ is $1$-strongly convex with respect to norm
  $\n{\cdot}$, we have $D(\bu,\bv)\geq \frac{1}{2}\n{\bu-\bv}^2$.

  Naturally, we shall say that $F\colon \dom g\to \cZ^*$ is
  $L_F$-Lipschitz, if $\n{F(\bu)-F(\bv)}_*\leq L_F \n{\bu-\bv}$,
  $\forall \bu,\bv$. However, Lipschitzness for a stochastic oracle
  this time will be more involved. Evidently, we prefer stochastic
  oracles $F_\xi$ of $F$ with as small $L$ as possible. Moreover, the
  proof of~\Cref{lem: eg1} indicates that in $k$-th iteration we need
  Lipschitzness only for already known two iterates. Hence, following~\cite{grigoriadis1995sublinear,carmon2019variance}, in contrast to Alg.~\ref{alg: eg1}, we will not
  fix distribution $Q$ in the beginning, but allow it to vary from iteration to
  iteration.  Formally, this amounts to the
  following definition.
\begin{definition}\label[definition]{def: def1}
  We say that $F$ has a stochastic oracle $F_\xi$ that is \emph{variable} $L$-Lipschitz in mean, if   for any
  $\bu,\bv\in \dom g$ there exists a distribution $Q_{\bu,\bv} $
  such that
  \begin{enumerate}[(i)]

  \item $F$ is unbiased: $F(\bz) = \E_{\xi\sim Q_{\bu,\bv}}\left[ F_{\xi}(\bz)\right]$\quad  $\forall \bz \in \dom g$;
  \item
    $\E_{\xi\sim Q_{\bu,\bv}}\left[\| F_\xi(\bu) - F_\xi(\bv) \|^2_*\right] \leq L^2 \| \bu-\bv \|^2$.
  \end{enumerate}
\end{definition}
Note that the second condition holds only for given $\bu,\bv$, but
  the constant $L$ is universal for all $\bu,\bv$.
Changing $\bu,\bv$
  also changes a distribution, hence the name ``variable''.  Without
  loss of generality, we denote any distribution that realizes the
  above Lipschitz bound for given $\bu$, $\bv$ by $Q_{\bu,\bv}$.
  This
  definition resembles the one in \cite[Definition 2]{carmon2019variance}.
  It is easy to see when $Q_{\bu, \bv}=Q$ for all
  $\bu, \bv$, we get the same definition as before in \Cref{asmp: asmp1}.

  We now introduce \Cref{asmp: asmp2} which will replace and generalize \Cref{asmp:
  asmp1}(iv).

\begin{assumption}\label{asmp: asmp2}~
The operator $F\colon \dom g \to \mathcal{Z}^\ast$ has a  stochastic  oracle $F_\xi$ that is variable $L$-Lipschitz in mean (see~\Cref{def: def1}).
\end{assumption}

\subsection{Mirror-Prox with variance reduction}
\begin{algorithm}[t]
\caption{Mirror-prox with variance reduction }
\label{alg:extra_ahmet}
\begin{algorithmic}[1]
    \STATE {\bfseries Input:} Step size
    $\tau$, $\a \in (0,1)$, $K>0$. Let $\bz^{-1}_j = \bz_0^0=\bw^0=\bz_0, \forall j\in[K]$
    \vspace{.2cm}
    \FOR{$s=0,1\dots$}
    \FOR{$k = 0,1\ldots K-1$}
    \STATE \label{mp:l1} $\bz_{k+1/2}^{s} = \argmin_{\bz}\Bigl\{ g(\bz) + \langle F(\bw^s) ,
    \bz \rangle + \frac{\a}{\tau} D(\bz, \bz_k^s) + \frac{1-\a}{\t}
    D(\bz, \mathbf{\w}^s)\Bigr\}$\label{eq:extra_double1}
    \STATE Fix distribution $Q_{\bz_{k+1/2}^s,\bw^s}$ and sample
    $\xi_k^s$ according to it
    \STATE $\Fhat(\bz_{k+1/2}^s) = F(\bw^s) +
        F_{\xi_k^s}(\bz_{k+1/2}^s) - F_{\xi_k^s}(\bw^s)$
    \STATE \label{mp:l2} $\bz_{k+1}^{s} = \argmin_{\bz}\Bigl\{ g(\bz) + \langle \Fhat(\bz_{k+1/2}^s), \bz \rangle +
        \frac{\a}{\tau} D(\bz, \bz_k^s) + \frac{1-\a}{\t} D(\bz,\mathbf{\w}^s)\Bigr\}$\label{eq:extra_loopless_double2}
        \ENDFOR
\STATE $\bw^{s+1} = \frac 1 K \sum_{k=1}^K \bz_k^s$\\[1mm]
\STATE $\grad(\mathbf{\w}^{s+1}) = \frac{1}{K} \sum_{k=1}^K \grad(\bz_k^s) $
\STATE $\bz^{s+1}_0 = \bz_K^s$
\ENDFOR
\end{algorithmic}
\label{alg: mp1}
\end{algorithm}
In this setting, we could simply adjust the steps of Alg.~\ref{alg:
  eg1} and correspondingly the analysis of~\Cref{lem: eg1}. However,
to show a convergence rate, double randomization in Alg.~\ref{alg: eg1}
causes technical complications. For this reason, in the Bregman setup
we propose a double loop variant of Alg.~\ref{alg: eg1}, similar to the classical
SVRG~\cite{johnson2013accelerating}. Our algorithm can be seen as a variant of Mirror-Prox~\cite{nemirovski2004prox} with variance reduction. Now it should
be clear that Alg.~\ref{alg: eg1} is a randomized version of Alg.~\ref{alg: mp1}
with $p = \frac 1 K$ and a particular choice $D(\bz,\bz')=\frac 12 \n{\bz-\bz'}^2_2$.

The technical reason for this change is the calculation given in~\eqref{eq: wje4}. In fact, all the other steps in the previous proofs would go through by using three point identity, except this step, which is inherently using the properties of $\ell_2$-norm.
By removing double randomization and introducing double loop instead, step~\eqref{eq: wje4} will not be needed in the analysis of Bregman case.

Compared to Alg.~\ref{alg: eg1}, $\bw^s$ serves the same purpose as $\bw_k$: the snapshot point in the language of SVRG~\cite{johnson2013accelerating}.
Since we have two loops in this case, we get $\bw^s$ by averaging, again, similar to SVRG for non-strongly convex optimization~\cite{reddi2016stochastic,allen2016improved}.
The difference due to Bregman setup is that we have the additional point $\bar \bw^s$ that averages in the dual space.
This operation does not incur additional cost.

\begin{remark}\label{rem: mp}
For running the algorithm in practice, we suggest $K=\frac{N}{2}$, $\a = 1
- \frac 1K$, and $\t = \frac{0.99\sqrt p}{L}$.
\end{remark}

\subsection{Analysis}
Similar to Euclidean case, we define for the iterates $(\bz_k^s)$
of Alg.~\ref{alg: mp1} and any $\bz\in \dom g$,
\begin{align*}
\Phi^{s}(\bz) &\coloneqq \a D(\bz, \bz_{0}^{s}) + (1-\a)\sum_{j=1}^K
D(\bz, \bz_j^{s-1}),
\end{align*}
      where $\Phi^0(\bz) = (\a + K(1-\a))D(\bz, \bz_0)$, due to the definition of $\bz^{-1}$ in Alg.~\ref{alg: mp1}.
      Since we have two indices $s, k$ in Alg.~\ref{alg: mp1}, we define $\mathcal{F}_k^s = \sigma(\bz_{1/2}^0, \ldots, \bz^0_{k-1/2}, \ldots, \bz^s_{1/2}, \ldots, \bz_{k+1/2}^s)$ and $\mathbb{E}_{s, k}[\cdot] = \mathbb{E}\left[\cdot \vert \mathcal{F}_k^s \right]$.

To analyze $\bz_{k+1}^s$ and $\bz_{k+1/2}^s$, we introduce the next lemma and provide its proof in~\Cref{sec: appendix}.
\begin{lemma}\label{lem:prox_update_br}
  Let $g$ be proper convex lsc, and
  \[\bz^+ = \argmin_\bz\left\{g(\bz) + \lr{\bu, \bz} + \a D(\bz, \bz_1) +
      (1-\a)D(\bz, \bz_2)\right\}.\]
  Then, for any $\bz$,
  \begin{align*}
g(\bz)-g(\bz^+) + \lr{\bu, \bz-\bz^+}\geq  D(\bz, \bz^+) +\a&\left( D(\bz^+, \bz_1) - D(\bz, \bz_1)  \right) \\
+ (1-\a) &\left(D(\bz^+, \bz_2) - D(\bz, \bz_2)  \right).
\end{align*}
\end{lemma}
\noindent We now introduce some definitions to be used in the proofs of this section.
\begin{align}
\Theta_{k+1/2}^s(\bz) &= \langle F(\bz_{k+1/2}^s), \bz_{k+1/2}^s -\bz\rangle + g(\bz_{k+1/2}^s) - g(\bz), \label{eq: def_theta_dl}\\
e(\bz, s, k) &= \tau \langle F(\bz_{k+1/2}^s) -
F_{\xi_k^s}(\bz_{k+1/2}^s) - F(\bw^s) + F_{\xi_k^s}(\bw^s),
\bz_{k+1/2}^s - \bz \rangle.\label{eq: def_e_dl}\\
\d(s, k) &= \tau \langle F_{\xi_k^s}(\bw^s) -
F_{\xi_k^s}(\bz_{k+1/2}^s), \bz_{k+1}^s - \bz_{k+1/2}^s
\rangle-\frac{1}{2}\| \bz_{k+1}^s -  \bz_{k+1/2}^s\|^2 -
\frac{1-\alpha}{2}\| \bz_{k+1/2}^s - \bw^s \|^2.\label{eq: delta}
\end{align}
The first expression will be needed for deriving the rate, the second
term $e(\bz, s,k)$ for controlling the error caused by
$\max_{\bz\in \cC}\E[\cdot]\neq \E \max_{\bz\in \cC}[\cdot]$, and the third term
$\d(s,k)$ will be nonpositive after taking expectation.
\begin{lemma}\label{lem: eg_loop}
Let~\Cref{asmp: asmp1} hold, $\alpha \in [0, 1)$, and
$\tau =\frac{\sqrt{1-\a}}{L}\gamma$ for $\gamma \in (0, 1)$. We have the following:
\begin{enumerate}[(i)]
\item For any $\bz \in \mathcal{Z}$ and $s, K\in\mathbb{N}$, it holds that
\begin{align*}
\sum_{k=0}^{K-1} \tau \Theta_{k+1/2}^s(\bz) &+ \alpha D(\bz, \bz_0^{s+1})+
(1-\alpha)\sum_{j=1}^K D(\bz, \bz_j^s)  \\ &\leq \alpha D(\bz, \bz_0^s) + (1-\alpha) \sum_{j=1}^K D(\bz, \bz_j^{s-1})
+\sum_{k=0}^{K-1}  [e(\bz, s, k) + \delta(s,k)].
\end{align*}
\item For any solution $\bz_\ast$, it holds that
\begin{equation*}
\mathbb{E}_{s, 0} \left[ \Phi^{s+1}(\bz_\ast) \right] \leq \Phi^s(\bz_\ast) - \frac{(1-\alpha)(1-\gamma^2)}{2} \sum_{k=0}^{K-1}  \mathbb{E}_{s, 0}\left[\|\bz_{k+1/2}^s - \bw^s \|^2\right].
\end{equation*}
\item It holds that $\sum_{s=0}^\infty \sum_{k=0}^{K-1} \E
    \|\bz_{k+1/2}^s - \bw^s \|^2  \leq \frac{2}{(1-\alpha)(1-\gamma^2)}\Phi^0(\bz_\ast)$.
\end{enumerate}
\end{lemma}
\begin{remark}
We use~\Cref{lem: eg_loop}(i) and~\Cref{lem: eg_loop}(iii) for proving the convergence rate.
On the other hand, \Cref{lem: eg_loop}(ii) can be used to derive subsequential convergence, which we do not include for brevity.
\end{remark}
\begin{proof}[Proof of~\Cref{lem: eg_loop}]
    Applying \Cref{lem:prox_update_br} to $\bz_{k+1/2}^s$ update, with $\bz = \bz_{k+1}^s$, we have
  \begin{multline}\label{eq:br-1}
\t \left(g(\bz_{k+1}^s)-g(\bz_{k+1/2}^s) + \lr{F(\bw^s),
    \bz_{k+1}^s-\bz_{k+1/2}^s}\right) \geq D(\bz_{k+1}^s, \bz_{k+1/2}^s)\\ +\a \Big(D(\bz_{k+1/2}^s, \bz_k^s) - D(\bz_{k+1}^s, \bz_k^s)  \Big)
+ (1-\a) \Big( D(\bz_{k+1/2}^s, \bwb^s) - D(\bz_{k+1}^s, \bwb^s)  \Big).
\end{multline}
Applying \Cref{lem:prox_update_br} to $\bz_{k+1}^s$ update with a general
$\bz\in \cZ$, we have
  \begin{multline}\label{eq:br-2}
\t \left(g(\bz)-g(\bz_{k+1}^s)  + \lr{\Fhat(\bz_{k+1/2}^s),
    \bz - \bz_{k+1}^s}\right) \geq D(\bz, \bz_{k+1}^s) \\
    + \a \Big(D(\bz_{k+1}^s, \bz_k^s) - D(\bz, \bz_k^s)  \Big)
+ (1-\a) \Big( D(\bz_{k+1}^s, \bwb^s) - D(\bz, \bwb^s)  \Big).
\end{multline}
Note that for any $\bu, \bv$, the expression $D(\bu, \bwb^s) - D(\bv,
\bwb^s)$ is linear in terms of $\nabla h(\bwb^s)$, that is
\begin{equation}\label{eq:simpler}
D(\bu, \bwb^s) - D(\bv, \bwb^s) =
  \frac{1}{K}\sum_{j=1}^K\left(D(\bu, \bz_j^{s-1}) - D(\bv,
    \bz_j^{s-1})\right).
\end{equation}
Summing up \eqref{eq:br-1} and \eqref{eq:br-2} and  using \eqref{eq:simpler} with definition of $\Fhat(\bz_{k+1/2}^s)$, we
obtain
\begin{align}
\tau\Bigl( g(\bz) - g(\bz_{k+1/2}^s) &+ \langle \Fhat(\bz_{k+1/2}^s), \bz - \bz_{k+1/2}^s
  \rangle \Bigr) \geq D(\bz, \bz_{k+1}^s) - \alpha D(\bz, \bz_k^s)
\notag \\
&+ \frac{1-\alpha}{K} \sum_{j=1}^K D(\bz_{k+1/2}^s, \bz_j^{s-1}) - \frac{1-\alpha}{K}\sum_{j=1}^K D(\bz, \bz_j^{s-1}) + D(\bz_{k+1}^s, \bz_{k+1/2}^s) \notag\\
&+ \tau \langle F_{\xi_k^s}(\bz_{k+1/2}^s) - F_{\xi_k^s}(\bw^s), \bz_{k+1}^s - \bz_{k+1/2}^s \rangle.\label{eq: mp1_lem_simp1}
\end{align}
By $D(\bu,\bv)\geq \frac 12 \n{\bu-\bv}^2$ and Jensen's inequality, we have
\begin{align}
\frac{1-\alpha}{K} \sum_{j=1}^K D(\bz_{k+1/2}^s, \bz_j^{s-1}) &\geq \frac{1-\alpha}{K}\sum_{j=1}^K \frac{1}{2} \| \bz_{k+1/2}^s - \bz_j^{s-1} \|^2 \geq \frac{1-\alpha}{2} \| \bz_{k+1/2}^s - \bw^s \|^2, \label{eq: breg_to_euc_1} \\
D(\bz_{k+1}^s, \bz_{k+1/2}^s) &\geq \frac{1}{2} \| \bz_{k+1}^s - \bz_{k+1/2}^s \|^2. \label{eq: breg_to_euc_2}
\end{align}
By using~\eqref{eq: def_theta_dl}, ~\eqref{eq: breg_to_euc_1}, and~\eqref{eq: breg_to_euc_2} in~\eqref{eq: mp1_lem_simp1}, we deduce
\begin{multline*}
\tau \Theta_{k+1/2}^s(\bz) + D(\bz, \bz_{k+1}^s) \leq \alpha D(\bz, \bz_k^s) + \frac{1-\alpha}{K}\sum_{j=1}^K D(\bz, \bz_j^{s-1}) \\
+\tau \langle F_{\xi_k^s}(\bw^s) - F_{\xi_k^s}(\bz_{k+1/2}^s), \bz_{k+1}^s - \bz_{k+1/2}^s \rangle -\frac{1}{2} \| \bz_{k+1}^s - \bz_{k+1/2}^s \|^2 - \frac{1-\alpha}{2}\| \bz_{k+1/2}^s -\bw^s \|^2 \\
+\underbrace{\tau \langle F(\bz_{k+1/2}^s) - \Fhat(\bz_{k+1/2}^s), \bz_{k+1/2}^s - \bz \rangle}_{e(\bz, s, k)},
\end{multline*}
where we defined the last term as $e(\bz, s, k)$ (see~\eqref{eq: def_e_dl}).
We sum this inequality over $k$ to obtain the result in $(i)$.

Next, similar to~\eqref{eq: eg1_ip3}, we estimate by~\Cref{asmp: asmp2}
and Young's  inequality
\begin{align}
\tau \mathbb{E}_{s, k} \langle F_{\xi_k^s}(\bw^s) - F_{\xi_k^s}(\bz_{k+1/2}^s)&, \bz_{k+1}^s - \bz_{k+1/2}^s \rangle \notag \\ &\leq \mathbb{E}_{s, k}\left[\frac{\tau^2}{2} \| F_{\xi_k^s}(\bw^s) - F_{\xi_k^s}(\bz_{k+1/2}^s) \|^2_\ast + \frac{1}{2} \| \bz_{k+1}^s - \bz_{k+1/2}^s \|^2\right] \notag \\
&\leq \frac{(1-\alpha)\gamma^2}{2} \| \bz_{k+1/2}^s-\bw^s \|^2 + \frac{1}{2} \mathbb{E}_{s, k}\| \bz_{k+1}^s - \bz_{k+1/2}^s \|^2,
\label{eq: mp1_4}
\end{align}
since $\tau^2L^2 = (1-\alpha)\gamma^2$.
We take expectation of~\eqref{eq: mp1_lem_simp1}, plug in $\bz=\bz_\ast$; use~\eqref{eq: eg1_ip2.5},~\eqref{eq: mp1_4},~\eqref{eq: breg_to_euc_1}, and~\eqref{eq: breg_to_euc_2} to get
\begin{equation}
\mathbb{E}_{s, k}\left[ D(\bz_\ast, \bz_{k+1}^s) \right] \leq \alpha
D(\bz_\ast, \bz_k^s) + \frac{1-\alpha}{K}\sum_{j=1}^K D(\bz_\ast,
\bz_j^{s-1})
+ \frac{(1-\alpha)(\gamma^2-1)}{2} \|\bz_{k+1/2}^s - \bw^s \|^2.\label{eq: mp1_5}
\end{equation}
By using $\mathbb{E}_{s, 0}[\cdot] =
\mathbb{E}_{s, 0} \left[ \mathbb{E}_{s, k}[\cdot]\right]$, we have
\begin{equation}
\mathbb{E}_{s, 0} D(\bz_\ast, \bz_{k+1}^s) \leq \E_{s,0}\Bigl[\alpha
D(\bz_\ast, \bz_k^s) + \frac{1-\alpha}{K}\sum_{j=1}^K D(\bz_\ast,
\bz_j^{s-1})
- \frac{(1-\alpha)(1-\gamma^2)}{2} \|\bz_{k+1/2}^s - \bw^s \|^2\Bigr].\label{eq: mp1_6}
\end{equation}
Summing this inequality over $k=0,\dots, K-1$ and using the definition
of $\Phi^s(\bz_*)$ together with $\bz^{s+1}_0 = \bz^s_K$, we  derive \emph{(ii)}.

 Finally, we take total expectation of \emph{(ii)} and sum the inequality over $s$ to obtain \emph{(iii)}.
\end{proof}

In order to prove the convergence rate, we need the Bregman version
of~\Cref{lem: exp_max_lem}. The proof of the lemma is in~\Cref{sec: appendix}.
\begin{lemma}\label{lem: exp_max_lem_br}
Let $\mathcal{F}=(\mathcal{F}_{k}^s)_{s\geq 0, k \in [0, K-1]}$ be a filtration and $(\bu_k^s)$ a stochastic process adapted to $\mathcal{F}$ with $\mathbb{E}[\bu_{k+1}^s | \mathcal{F}_{k}^s] = 0$. Given $
\bx_0\in\mathcal{Z}$,
for any $S \in \mathbb{N}$ and any compact set
$\mathcal{C}\subset\dom g$
\begin{equation*}
\mathbb{E}\left[\max_{\bx\in\mathcal{C}} \sum_{s=0}^{S-1}\sum_{k=0}^{K-1} \langle \bu_{k+1}^s, \bx \rangle\right] \leq \max_{\bx\in\mathcal{C}}D(\bx,  \bx_0) +\frac{1}{2} \sum_{s=0}^{S-1}\sum_{k=0}^{K-1} \E\| \bu_{k+1}^s \|^2_{\ast}.
\end{equation*}
\end{lemma}
We now continue with the main result of this section.
\begin{theorem}\label{eq: eg1_th_br}
 Let~\Cref{asmp: asmp1}(i,ii,iii) and~\Cref{asmp: asmp2} hold, $\alpha \in[0, 1)$, and
  $\t =\frac{\sqrt{1-\a}}{L}\gamma$ for $\gamma\in(0,1)$. Then,
  for $\bz^{S} = \frac{1}{KS}\sum_{s=0}^{S-1}\sum_{k=0}^{K-1} \bz_{k+1/2}^s$, it follows that
\begin{equation*}
\mathbb{E}\left[ \mathrm{Gap}(\bz^{S}) \right] \leq \frac{\text{1} }{\t KS} \left(1+\Bigl(1+
  \frac{8\gamma^2}{1-\gamma^2}\Bigr)(\a + K(1-\a) \right)\max_{\bz\in\cC}  D(\bz, \bz_0).
\end{equation*}
\end{theorem}
\begin{corollary}\label{cor:compl_br}
    Let $K = \frac{N}{2}$ and $\a = 1 - \frac{1}{K} = 1 -
    \frac{2}{N}$, and $\t = \frac{\sqrt{1-\a}}{L} \gamma$ for $\gamma\in(0,1)$. Then the
  total complexity of Alg.~\ref{alg: mp1} to reach $\e$-accuracy is
   $\cO \left(N+ \frac{L\sqrt N}{ \e} \right)$.
In particular, if $\t = \frac{\sqrt{1-\a}}{3L} = \frac{\sqrt 2}{3\sqrt
  N L} $, the total complexity is $2N+\frac{43\sqrt N L}{\e}\max_{\bz\in\cC}  D(\bz, \bz_0)$.
\end{corollary}
\begin{proof}[Proof of~\Cref{eq: eg1_th_br}]

We start with the result of~\Cref{lem: eg_loop} and proceed similar to~\Cref{eq: eg1_th2}.
Since $\bz_0^{s+1} = \bz_K^s$, we use definition of $\Phi^s(\bz)$, and sum
the inequality in~\Cref{lem: eg_loop}(i) over $s$ to obtain
\begin{equation*}
\sum_{s=0}^{S-1} \sum_{k=0}^{K-1} \tau \Theta_{k+1/2}^s(\bz) +
\Phi^S(\bz) \leq \Phi^0(\bz) +\sum_{s=0}^{S-1}\sum_{k=0}^{K-1} [e(\bz,
s, k)+\d(s,k)]
\end{equation*}
We take maximum and expectation, use $\mathbb{E}\left[\max_{\bz\in\cC} \sum_{s=0}^{S-1} \sum_{k=0}^{K-1} \tau\Theta_{k+1/2}^s(\bz) \right]\geq \tau KS \mathbb{E} \left[\mathrm{Gap}(\bz^S)\right]$ to deduce
\begin{equation*}
\tau KS \mathbb{E}\left[ \mathrm{Gap}(\bz^S) \right] \leq \max_{\bz\in
  \cC}\Phi^0(\bz) +\mathbb{E}\left[\max_{\bz\in\mathcal{C}} \sum_{s=0}^{S-1}
\sum_{k=0}^{K-1} e(\bz, s, k)\right]+ \mathbb{E}\left[\sum_{s=0}^{S-1} \sum_{k=0}^{K-1} \d( s, k)\right].
\end{equation*}
The term $\E\sum_{s=0}^{S-1} \sum_{k=0}^{K-1}\d(s,k)$ is nonpositive by the tower property,
Lipschitzness, Young's inequality, and $\t < \frac{\sqrt{p}}{L}$
(the same arguments used in~\eqref{eq: mp1_4} can be applied here
with $\delta(s,k)$ defined as \eqref{eq: delta}). Therefore,
\begin{equation*}
\tau KS \mathbb{E}\left[ \mathrm{Gap}(\bz^S) \right] \leq \max_{\bz\in\cC} \Phi^0(\bz) + \mathbb{E}\left[\max_{\bz\in\mathcal{C}} \sum_{s=0}^{S-1}
\sum_{k=0}^{K-1} e(\bz, s, k)\right].
\end{equation*}
We bound the second term on RHS, similar to the proof of~\Cref{eq: eg1_th2}.
For $s\in\{0,\ldots,S-1\}$ and $k\in\{0,\ldots,K-1\}$, set $ \mathcal{F}_k^s = \sigma(\bz_{1/2}^0, \ldots, \bz^0_{K-1/2}, \ldots, \bz^s_{1/2}, \ldots, \bz_{k+1/2}^s)$, $\bu_{k+1}^s= \tau[\Fhat(\bz_{k+1/2}^s) - F(\bz_{k+1/2}^s)] = \tau[F(\bw^s) - F_{\xi_k^s}(\bw^s) - F(\bz_{k+1/2}^s) + F_{\xi_k^s}(\bz_{k+1/2}^s)]$, which help us write
\begin{align*}
\mathbb{E}\Bigg[\max_{\bz\in\cC} \sum_{s=0}^{S-1}\sum_{k=0}^{K-1}e(\bz, k)\Bigg] &= \mathbb{E}\left[\max_{\bz\in\cC} \sum_{s=0}^{S-1} \sum_{k=0}^{K-1} \tau \langle \Fhat(\bz_{k+1/2}^s) - F(\bz_{k+1/2}^s), \bz - \bz_{k+1/2}^s \rangle \right] \\
&= \mathbb{E}\left[\max_{\bz\in\cC}\sum_{s=0}^{S-1}\sum_{k=0}^{K-1} \langle \bu_{k+1}^s, \bz \rangle \right] - \sum_{s=0}^{S-1}\sum_{k=0}^{K-1} \mathbb{E}\langle \bu_{k+1}^s, \bz_{k+1/2}^s \rangle \\
&= \mathbb{E}\left[\max_{\bz\in\cC}\sum_{s=0}^{S-1}\sum_{k=0}^{K-1} \langle \bu_{k+1}^s, \bz \rangle\right],
\end{align*}
where the last equality is due to the tower property, $\mathcal{F}_k^s$-measurability of $\bz_{k+1/2}^s$ and $\mathbb{E}_{s, k}[\bu_{k+1}^s] = 0$.

We apply~\Cref{lem: exp_max_lem_br} with the specified $\mathcal{F}_k^s$, $\bu_{k+1}^s$ to obtain
\begin{align}
\mathbb{E} \bigg[\max_{\bz\in\cC} \sum_{s=0}^{S-1} &\sum_{k=0}^{K-1} e(\bz, k)\bigg] \notag \\
&\leq \max_{\bz\in\cC} D(\bz, \bz_0) + \sum_{s=0}^{S-1} \sum_{k=0}^{K-1} \tau^2 \mathbb{E} \| F_{\xi_k^s}(\bz_{k+1/2}^s)-F_{\xi_k^s}(\bw^s) + F(\bw^s) -F(\bz^s_{k+1/2}) \|^2_\ast \notag \\
&\leq \max_{\bz\in\cC} D(\bz, \bz_0) + \sum_{s=0}^{S-1} \sum_{k=0}^{K-1} 4\tau^2 \mathbb{E} \| F_{\xi_k^s}(\bz_{k+1/2}^s)-F_{\xi_k^s}(\bw^s) \|^2_\ast \label{eq: mp1_rate_3}\\
&\leq \max_{\bz\in\cC} D(\bz, \bz_0) + \sum_{s=0}^{S-1} \sum_{k=0}^{K-1}4\tau^2L^2 \mathbb{E} \| \bz_{k+1/2}^s-\bw^s \|^2 \label{eq: mp1_rate_4}\\
&\leq \max_{\bz\in\cC} D(\bz, \bz_0) + \frac{8\tau^2L^2}{(1-\alpha)(1-\gamma^2)} \Phi^0(\bz_\ast),\label{eq:br_is_not_eucl}
\end{align}
where~\eqref{eq: mp1_rate_3} is due to the tower property and $\mathbb{E}\| X - \mathbb{E} X \|^2_\ast \leq 2\mathbb{E}\|X\|^2_\ast + 2\| \mathbb{E}X \|^2_\ast \leq 4\mathbb{E}\| X\|^2_\ast$, which follows from triangle inequality, Young's inequality, and Jensen's inequality.
  Moreover,~\eqref{eq: mp1_rate_4} is by variable Lipschitzness of $F_{\xi}$, and the last step is by~\Cref{lem: eg_loop}.
Consequently, by $\Phi^0(\bz_\ast) \leq
\max_{\bz\in\cC}\Phi^0(\bz) = (\a + K(1-\a))\max_{\bz\in \cC}D(\bz, \bz_0)$ and $\tau^2L^2 = (1-\alpha)\gamma^2$ we have
\begin{align*}
\tau KS \mathbb{E}\left[ \mathrm{Gap}(\bz^S) \right] &\leq
\max_{\bz\in\cC}\left[  D(\bz, \bz_0) + \Bigl(1+
  \frac{8\t^2L^2}{(1-\alpha)(1-\gamma^2)}\Bigr)\Phi^0(\bz)\right]\\
& =\left(1+\Bigl(1+
  \frac{8\gamma^2}{1-\gamma^2}\Bigr)(\a + K(1-\a) \right)\max_{\bz\in\cC}  D(\bz, \bz_0),
\end{align*}
which gives the result.
\end{proof}
\begin{proof}[Proof of~\Cref{cor:compl_br}]
    As
  $\a = 1- \frac 1 K$, it holds that
  $\a + K(1-\a) = 1 - \frac 1 K + 1\leq 2$.  With this, from
  \Cref{eq: eg1_th_br} it follows
\begin{align}
  \mathbb{E}\left[ \mathrm{Gap}(\bz^S) \right]&\leq
  \frac{1}{\t K S}\left(1+\Bigl(1+
  \frac{8\gamma^2}{1-\gamma^2}\Bigr)(\a + K(1-\a) \right)\max_{\bz\in\cC}
D(\bz, \bz_0)\notag\\
&\leq
 \frac{L}{\sqrt{K}\gamma S}\left(3+
  \frac{16\gamma^2}{1-\gamma^2}
\right)\max_{\bz\in\cC}  D(\bz, \bz_0)  = \cO \left( \frac{L}{\sqrt N S} \right).\label{eq:br_rate}
\end{align}
   One epoch requires one evaluation of $F$
and $2K$ of $F_\xi$, therefore in total we have
$N + 2K = 2N $. To reach $\e$ accuracy, we need
$\left\lceil\cO\left(\frac{L}{\sqrt N \e}\right) \right\rceil$
epochs. Hence, the final complexity is
$\cO\left(N+\frac{L\sqrt N}{\e}\right)$.  Now, by setting
$\gamma = \frac 1 3$ in \eqref{eq:br_rate}, we will get specific
constants. In particular, we will have
\begin{align*}
  \mathbb{E}\left[ \mathrm{Gap}(\bz^S) \right]
&\leq
 \frac{15L}{\sqrt{K}S}\max_{\bz\in\cC}  D(\bz, \bz_0) =  \frac{15\sqrt
   2L}{\sqrt{N}S}\max_{\bz\in\cC}  D(\bz, \bz_0).
\end{align*}
Consequently, since $30\sqrt 2<43$, the final complexity is
$\left(2N + \frac{43\sqrt{N}L}{\e} \max_{\bz\in\cC}  D(\bz,
  \bz_0)\right)$.
\end{proof}
\begin{remark}
  Because we work with general norms, we had to use in
  \eqref{eq:br_is_not_eucl} a crude inequality
  $\mathbb{E}\| X - \mathbb{E} X \|^2_\ast \leq 4\mathbb{E}\|
  X\|^2_\ast$. Of course, in the Euclidean case with $D(\bz,\bz')=\frac{1}{2}\n{\bz-\bz'}^2$ this factor $4$ is
  redundant. It is easy to see that setting
  $\t = \frac{\sqrt{1-\a}}{2L}$ and the rest of the parameters as in
  \Cref{cor:compl_br} leads to
  $\left(2N + \frac{13\sqrt{N}L}{\e}\max_{\bz\in\cC} \n{\bz-
    \bz_0}^2\right)$ total complexity for the Euclidean setting.
  \end{remark}

\section{Extensions}\label{sec: extensions}
In this section, we show how to obtain the variance reduced versions of two other  operator splitting methods: forward-backward-forward (FBF)~\cite{tseng2000modified} and forward-reflected-backward (FoRB)~\cite{malitsky2018forward} for monotone inclusions.
We also show how to obtain linear convergence with Algorithm~\ref{alg: eg1} when $g$ in~\eqref{eq: prob_vi} is strongly convex.

Formally, the monotone inclusion problem is to find
\begin{equation}\label{eq: monot_inc}
 \bz_\ast \in \mathcal{Z}\text{ such that } 0\in(F+G)(\bz_\ast),
\end{equation}
where $\mathcal{Z}$ is a finite dimensional vector space with Euclidean inner product and the rest of the assumptions are summarized in~\Cref{asmp: asmp3}.
\begin{assumption}\label{asmp: asmp3}~
\begin{enumerate}[(i)]
\item The solution set $\mathsf{Sol}$ of~\eqref{eq: monot_inc} is nonempty: $(F+G)^{-1}(0) \neq \empty$.
\item The operators $G\colon\mathcal{Z}\rightrightarrows\mathcal{Z}$ and $F\colon\mathcal{Z}\to\mathcal{Z}$ are maximally monotone.
\item The operator $F$ has an oracle $F_\xi$ that
  is unbiased $F(\bz)=\mathop{\mathbb{E}}_{\xi}\left[ F_\xi(\bz) \right]$ and $L$-Lipschitz in mean:
\begin{equation*}
{\mathbb{E}}_{\xi}\left[\| F_\xi(\bu) - F_\xi(\bv) \|^2\right] \leq L^2 \| \bu-\bv \|^2,~~~~~\forall \bu, \bv\in\mathcal{Z}.
\end{equation*}
\end{enumerate}
\end{assumption}
We remark that one can use variable Lipschitz assumption from~\Cref{asmp: asmp2} instead of standard Lipschitzness, but we chose the latter for simplicity.
Let us also recall the conditional expectation definitions based on the iterates of the algorithms: $\mathbb{E}[\cdot | \sigma(\xi_0, \dots, \xi_{k-1},\bw_k)] = \mathbb{E}_k[\cdot]$ and $\mathbb{E}[\cdot | \sigma(\xi_0, \dots, \xi_{k},\bw_k] = \mathbb{E}_{k+1/2}[\cdot]$.
Next, the resolvent of an operator $G$ is given by $J_G = (I +
G)^{-1}$ where $I$ is the identity operator. It is easy to see that
when $G=\partial g$ for proper convex lsc function $g$,
inclusion~\eqref{eq: monot_inc} becomes the VI in~\eqref{eq: prob_vi} and $J_G=\prox_g$.

\subsection{Forward-Backward-Forward with variance reduction}
\begin{algorithm}[t]
\caption{FBF with variance reduction }
\begin{algorithmic}[1]
    \STATE {\bfseries Input:} Probability $p\in (0,1]$, probability
    distribution $Q$, step size
    $\tau$, $\a \in (0,1)$. Let $\mathbf{z}_0=\mathbf{w}_0$
    \vspace{.2cm}
    \FOR{$k = 0,1\ldots $}
        \STATE $\bzb_k = \a \bz_k + (1-\a)\bw_k$
        \STATE $\bz_{k+1/2} = J_{\tau G}(\bzb_k - \tau F(\bw_k))$
            \STATE Draw an index $\xi_k$ according to $Q$
        \STATE $\bz_{k+1} = \bz_{k+1/2} - \tau(F_{\xi_k}(\bz_{k+1/2}) - F_{\xi_k}(\bw_k))$
        \STATE $\bw_{k+1} = \begin{cases}
                \bz_{k+1}, \text{~~with probability~~} p\\
                \bw_k, \text{~~~~with probability~~} 1-p
            \end{cases}$
        \ENDFOR
\end{algorithmic}
\label{alg: fbf}
\end{algorithm}

Forward-backward-forward (FBF) algorithm was introduced by Tseng
in~\cite{tseng2000modified}. On one hand, it is a modification of the forward-backward
algorithm that does not require stronger assumptions than mere
monotonicity. On the other, it is a modification of the extragradient
method that works for general monotone inclusions and not just for
variational inequalities. FBF reads as
\begin{equation*}
\begin{cases}
\bz_{k+1/2} = J_{\tau G}(\bz_k - \tau F(\bz_k)) \\
\bz_{k+1} = \bz_{k+1/2} - \tau F(\bz_{k+1/2}) + \tau F(\bz_k).
\end{cases}
\end{equation*}
It is easy to see that FBF is equivalent to extragradient when $G$ is
absent.  But when not,  FBF applied to the VI requires one
proximal operator every iteration, whereas extragradient
requires two.  This advantage can be important for the cases where
proximal operator is computationally expensive~\cite{bohm2020two}.

\begin{remark}\label{rem: fbf}
For running Alg.~\ref{alg: fbf} in practice, we suggest $p=\frac{2}{N}$, $\a = 1 - p$, and $\t = \frac{0.99\sqrt p}{L}$.
\end{remark}
We keep the same notation as~\Cref{subs:eucl} and recall the definition of $\Phi_k$ for convenience
\begin{equation*}
\Phi_k(\bz) = \alpha \| \bz_k - \bz \|^2 + \frac{1-\alpha}{p} \| \bw_k - \bz\|^2.
\end{equation*}
We now continue with the main result for FBF.
\begin{theorem}\label{th: fbf}
Let~\Cref{asmp: asmp3} hold, $\alpha \in [0, 1)$, $p\in(0, 1]$, and
$\tau =\frac{\sqrt{1-\a}}{L}\gamma$ for $\gamma \in(0,1)$. Then for $(\bz_k)$ generated by Alg.~\ref{alg: fbf} and any $\bz_\ast \in\Sol$, it holds that
\begin{equation*}
\mathbb{E}_k \left[ \Phi_{k+1}(\bz_\ast) \right] \leq \Phi_k(\bz_\ast).
\end{equation*}
Moreover, if $F_\xi$ is continuous for all $\xi$, then $(\bz_k)$ converges to
some $\bz_\ast\in \Sol$ a.s.
\end{theorem}
\begin{proof}
Let $\bz=\bz_\ast\in \Sol$ which gives $-F(\bz) \in G(\bz)$.
Next, by the definition of $\bz_{k+1/2}$ and resolvent, $\bzb_k - \tau F(\bw_k) \in\bz_{k+1/2}+\tau G(\bz_{k+1/2})$.
Combining these estimates with monotonicity of $G$ lead to
\begin{align*}
\langle \bz_{k+1/2} - \bzb_k + \tau F(\bw_k), \bz -\bz_{k+1/2} \rangle - \tau \langle F(\bz), \bz - \bz_{k+1/2} \rangle \geq 0.
\end{align*}
We plug in the definition of $\bz_{k+1}$ into this inequality to obtain
\begin{align}
\langle \bz_{k+1} - \bzb_k + \tau \left( F_{\xi_k}(\bz_{k+1/2}) - F_{\xi_k}(\bw_k) + F(\bw_k) \right),\bz - \bz_{k+1/2} \rangle - \langle F(\bz), \bz - \bz_{k+1/2} \rangle \geq 0.\label{eq: fbf_eq1}
\end{align}
We estimate  the term with $\bzb_k$  as in~\eqref{eq: eg1_ip1}
\begin{align}
2\langle \bz_{k+1} - \bzb_k&, \bz-\bz_{k+1/2} \rangle = 2\langle \bz_{k+1} - \bz_{k+1/2}, \bz-\bz_{k+1/2} \rangle + 2\langle \bz_{k+1/2} - \bzb_k, \bz-\bz_{k+1/2} \rangle \notag \\
&= \| \bz_{k+1} - \bz_{k+1/2} \|^2 + \| \bz-\bz_{k+1/2}\|^2 - \| \bz-\bz_{k+1} \|^2 + 2\langle \bz_{k+1/2} - \bzb_k, \bz-\bz_{k+1/2} \rangle \notag \\
&= \| \bz_{k+1} - \bz_{k+1/2} \|^2 - \| \bz-\bz_{k+1} \|^2 + \alpha \| \bz-\bz_k \|^2  + (1-\alpha) \| \bw_k - \bz\|^2 \notag \\
&- \alpha \| \bz_{k+1/2} - \bz_k\|^2 - (1-\alpha) \| \bz_{k+1/2} - \bw_k \|^2.\label{eq: fbf_eq2}
\end{align}
By taking conditional expectation and using that $\bz_{k+1/2}$ is $\mathcal{F}_k$-measurable, we deduce
\begin{align}
2\tau\mathbb{E}_k \left[ \langle F_{\xi_k}(\bz_{k+1/2}) - F_{\xi_k}(\bw_k) + F(\bw_k), \bz-\bz_{k+1/2} \rangle \right] = 2\tau\mathbb{E}_k\left[ \langle F(\bz_{k+1/2}), \bz-\bz_{k+1/2} \rangle \right].\label{eq: fbf_eq3}
\end{align}
We use~\eqref{eq: fbf_eq2} and~\eqref{eq: fbf_eq3} in~\eqref{eq: fbf_eq1} to obtain
\begin{multline*}
2\tau \langle F(\bz) - F(\bz_{k+1/2}), \bz - \bz_{k+1/2} \rangle + \mathbb{E}_k\|\bz_{k+1} -\bz\|^2 \leq \alpha\| \bz_k - \bz \|^2 + (1-\alpha) \| \bw_k -\bz\|^2 \\
+ \mathbb{E}_k\| \bz_{k+1} - \bz_{k+1/2} \|^2
- \alpha \| \bz_{k+1/2} - \bz_k \|^2 - (1-\alpha) \| \bz_{k+1/2} - \bw_k \|^2.
\end{multline*}
Note that, the first term in the LHS is nonnegative by monotonicity of
$F$. Then we add~\eqref{eq: eg1_wexp} to this inequality and use $ \| \bz_{k+1} - \bz_{k+1/2} \|^2 \leq \tau^2 L^2 \|\bz_{k+1/2} - \bw_k \|^2$ to obtain
\begin{align*}
\alpha \mathbb{E}_k \| \bz_{k+1} - \bz \|^2 + \frac{1-\alpha}{p} \| \bw_{k+1} - \bz \|^2 \leq \alpha\| \bz_{k} - \bz \|^2 &+ \frac{1-\alpha}{p} \| \bw_k - \bz \|^2 - \alpha \| \bz_{k+1/2} - \bz_k \|^2 \\
&- \left( (1-\alpha) - \tau^2 L^2 \right) \| \bz_{k+1/2} - \bw_k\|^2.
\end{align*}
This derives the first result, which is the analogue of~\Cref{lem: eg1}.
To show almost sure convergence, we basically follow the proof of~\Cref{th: eg1_th1}.
First, using Robbins-Siegmund theorem and~\cite[Proposition
2.3]{combettes2015stochastic} as in~\Cref{th: eg1_th1}, we obtain that
there exists a probability $1$ set $\Xi$ of random trajectories such
that $\forall \theta\in\Xi$ and $\forall \bz \in \Sol$, we have that $\alpha\|  \bz_{k}(\theta) -
\bz \|^2 + \frac{1-\alpha}{p}\| \bw_k(\theta) - \bz\|^2$ converges, $\bz_{k+1/2}(\theta)
- \bz_k(\theta) \to 0$, and $\bz_{k+1/2}(\theta) - \bw_k(\theta) \to
0$. The latter implies $ \bz_{k+1}(\theta) - \bz_{k+1/2}(\theta) \to
0$.
Let $\tilde \bz(\theta)$ be a cluster point of the bounded sequence $(\bz_k(\theta))$.
Instead of~\eqref{eq: fix_pt_vi}, we use the definitions of $\bz_{k+1/2}$, resolvent, and $\bz_{k+1}$ to obtain
\begin{multline}\label{eq: fix_pt_fbf}
\bz_{k+1}(\theta) - \bzb_k(\theta) + \tau \left( F_{\xi_k}(\bw_k(\theta)) - F_{\xi_k}(\bz_{k+1/2}(\theta)) \right) + \tau \left( F(\bz_{k+1/2}(\theta)) - F(\bw_k(\theta)) \right) \\
\in \tau \left( F+G \right)(\bz_{k+1/2}(\theta)),
\end{multline}
to show that $\tilde \bz(\theta) \in (F+G)^{-1}(0)$.
In particular, we use that $F_\xi$ is continuous for all $\xi$,
$\bz_{k+1} - \bzb_k \to 0$, and $\bz_{k+1/2} - \bw_k \to 0$ almost
surely.
We use the same arguments as the proof of~\Cref{th: eg1_th1} to conclude.
\end{proof}

We next give the complexity of the algorithm for solving VI as~\Cref{sec: euc_sub_rate}.
The derivation is essentially the same as~\Cref{sec: euc_sub_rate} and therefore omitted.
  \begin{corollary}\label[corollary]{cor: fbf}
Let $\alpha = 1-p = 1 - \frac{2}{N}$ and $\bz^K = \frac{1}{K} \sum_{k=0}^{K-1} \bz_{k+1/2}$. Then, the total complexity to get an $\varepsilon$-accurate solution to~\eqref{eq: prob_vi} is $\mathcal{O}\left( N+\frac{\sqrt{N}L}{\epsilon}\right)$.
\end{corollary}

\subsection{Forward-reflected-backward with variance reduction: revisited}
In a similar spirit to FBF, but using a different idea,
\cite{malitsky2018forward} proposed FoRB method
\[\bz_{k+1} = J_{\tau G}\left(\bz_k - \tau[F(\bz_k) + F(\bz_k) -
  F(\bz_{k-1})]\right).\]
This scheme generalizes optimistic gradient descent~\cite{rakhlin2013online,daskalakis2018training} and in some
particular cases is equivalent to Popov's method~\cite{popov1980modification}.
Later, in~\cite{alacaoglu2020simple}, the authors suggested the most straightforward
variance reduction modification of FoRB by combining FoRB and loopless SVRG~\cite{kovalev2020dont}.
This algorithm had the drawback of small step sizes which lead to complexity bounds that do not improve upon the deterministic methods.
As highlighted in the experiments of~\cite{alacaoglu2020simple}, the small step size $\tau\sim\frac{1}{n}$ seemed to be non-improvable for the given method.
One possible speculation for this phenomenon might be that the
method is too aggressive and therefore prohibits large step
sizes.
We will use the retracted iterate $\bzb_k = \alpha \bz_k + (1-\alpha)\bw_k$ instead of the latest iterate $\bz_k$ in the update to improve complexity.

The advantage of FoRB compared to extragradient is similar to FBF.
FoRB only needs one proximal operator, applied to VI.  Compared to
FBF, FoRB has a simpler update rule and, unlike FBF, it is easy to
adjust to Bregman setting, see
\cite{alacaoglu2020simple,zhang2021extragradient}.

\begin{algorithm}[t]
\caption{FoRB with variance reduction}
\begin{algorithmic}[1]
    \STATE {\bfseries Input:} Probability $p\in (0,1]$, probability
    distribution $Q$, step size
    $\tau$, $\a \in (0,1)$. Let $\mathbf{z}_0=\mathbf{w}_0$
    \vspace{.2cm}
    \FOR{$k = 1,2\ldots $}
        \STATE $\bzb_k = \a \bz_k + (1-\a)\bw_k$
            \STATE Draw an index $\xi_k$ according to $Q$
        \STATE $\bz_{k+1} = J_{\tau G} (\bzb_k -\t  F(\bw_k) -  \tau
        (F_{\xi_k}(\bz_k) - F_{\xi_k}(\bw_{k-1})))$
        \STATE $\bw_{k+1} = \begin{cases}
                \bz_{k+1}, \text{~~with probability~~} p\\
                \bw_k, \text{~~~~with probability~~} 1-p
            \end{cases}$
        \ENDFOR
\end{algorithmic}
\label{alg: forb}
\end{algorithm}

\begin{remark}\label{rem: forb}
For running Alg.~\ref{alg: forb} in practice, we suggest $p=\frac{2}{N}$, $\a = 1 - p$, and $\t = \frac{0.99\sqrt{p(1-p)}}{L}$.
\end{remark}
Lyapunov function here is slightly more complicated than the ones in previous sections:
\begin{equation*}
\Phi_{k+1}(\bz) \coloneqq \a \| \bz_{k+1} - \bz \|^2 + \frac{1-\a}{p}\n{ \bw_{k+1} - \bz}^2 + 2\t \langle F(\bz_{k+1}) - F(\bw_k), \bz- \bz_{k+1} \rangle + (1-\a) \| \bz_{k+1} - \bw_k \|^2.
\end{equation*}

\begin{theorem}\label{th: forb}
Let~\Cref{asmp: asmp3} hold, $\alpha \in [0, 1)$, $p\in(0, 1]$, and
$\tau=\frac{\sqrt{\alpha(1-\a)}}{L}\gamma$ for $\gamma \in(0,1)$. Then for $(\bz_k)$ generated by Alg.~\ref{alg: forb} and any
$\bz_\ast \in\Sol$, it holds that $\Phi_k(\bz_\ast)$ is nonnegative
and
\begin{equation*}
\mathbb{E}_k \left[ \Phi_{k+1}(\bz_\ast) \right] \leq \Phi_k(\bz_\ast).
\end{equation*}
Moreover, if $F_\xi$ is continuous for all $\xi$, then  $(\bz_k)$ converges to
some $\bz_\ast\in\Sol$ a.s.
\end{theorem}
\begin{remark}
  Note that again when randomness is null, $F_\xi = F$ and $p=1$,
  Alg.~\ref{alg: forb} reduces to the original FoRB algorithm. Moreover,
  with $\a = \frac 1 2$ we recover the result in \cite{malitsky2018forward}.
\end{remark}

\begin{proof}[Proof of~\Cref{th: forb}]
Nonnegativity of $\Phi_k(\bz_*)$ is straightforward to prove by using
Lipschitzness of $F$ and $\t L\leq \sqrt{a (1-\a)}$.

Let $\bz=\bz_\ast\in \Sol$ which gives $-F(\bz) \in G(\bz)$.
Next, by the definitions of $\bz_{k+1}$ and resolvent, $\bzb_k - \tau \left[F(\bw_k) - F_{\xi_k}(\bw_{k-1}) + F_{\xi_k}(\bz_k) \right] \in \bz_{k+1}+\tau G(\bz_{k+1})$. Combining these estimates and monotonicity of $G$ leads to
\begin{align}
\langle \bz_{k+1} - \bzb_k + \tau \left[F(\bw_k) - F_{\xi_k}(\bw_{k-1}) + F_{\xi_k}(\bz_k) \right], \bz -\bz_{k+1} \rangle - \tau \langle F(\bz), \bz - \bz_{k+1} \rangle \geq 0.\label{eq: forb_1}
\end{align}
  We split the first inner product and work with each term separately. First,
  \begin{align*}
\t\langle F(\bw_k) -F_{\xi_k}(&\bw_{k-1}) + F_{\xi_k}(\bz_k), \bz-\bz_{k+1}\rangle \\&= \t
    \lr{F(\bw_k)-F(\bz_{k+1}),\bz-\bz_{k+1}}  - \t
    \lr{F_{\xi_k}(\bw_{k-1})-F_{\xi_k}(\bz_{k}),\bz-\bz_{k+1}} \\
    &\quad + \t \lr{F(\bz_{k+1}), \bz-\bz_{k+1}}\\ &=\t
    \lr{F(\bw_k)-F(\bz_{k+1}),\bz-\bz_{k+1}}  - \t \lr{F_{\xi_k}(\bw_{k-1})-F_{\xi_k}(\bz_{k}),\bz-\bz_{k}} \\ &\quad
    -\t \lr{F_{\xi_k}(\bw_{k-1})-F_{\xi_k}(\bz_{k}),\bz_{k}-\bz_{k+1}} + \t \lr{F(\bz_{k+1}), \bz-\bz_{k+1}}.
  \end{align*}
Second, as we derived in~\eqref{eq: eg1_ip1},
\begin{align*}
   2\langle \bz_{k+1}-\bzb_k, \bz - \bz_{k+1}\rangle =\alpha \| \bz_k - \bz\|^2 &- \| \bz_{k+1} - \bz \|^2 + (1-\alpha) \| \bw_k - \bz\|^2 \\&
   - \alpha \| \bz_{k+1} - \bz_k \|^2 - (1-\alpha) \| \bz_{k+1}-\bw_k \|^2.
    \end{align*}
    Substituting the last two estimates into \eqref{eq: forb_1}, we obtain
    \begin{align}
      \label{eq:1}
& \n{\bz_{k+1}-\bz}^2 + 2\t \lr{F(\bz_{k+1})-F(\bw_k),\bz-\bz_{k+1}} + 2\t\langle F(\bz) - F(\bz_{k+1}), \bz - \bz_{k+1} \rangle \notag \\ \leq \a &\n{\bz_k-\bz}^2 + (1-\a)\n{\bw_k-\bz}^2 +2\t
 \lr{F_{\xi_k}(\bz_{k})-F_{\xi_k}(\bw_{k-1}),\bz-\bz_{k}} \notag \\&+2\t \lr{F_{\xi_k}(\bz_{k})-F_{\xi_k}(\bw_{k-1}),
   \bz_{k}-\bz_{k+1}} - \a\n{\bz_{k+1}-\bz_k}^2 - (1-\a)\n{\bz_{k+1}-\bw_k}^2.
\end{align}
We take expectation conditioning on the knowledge of $\bz_k, \bw_k$,
use $\mathbb{E}_k  F_{\xi_k}(\bz_k) =  F(\bz_k)$, $\mathbb{E}_k
F_{\xi_k}(\bw_{k-1}) = F(\bw_{k-1})$, and monotonicity of $F$ for the
third term in the LHS. This yields
    \begin{align}
      \label{eq: forb_2}
      &\E_k \left[\n{\bz_{k+1}-\bz}^2 + 2\t \lr{F(\bz_{k+1})-F(\bw_k),\bz-\bz_{k+1}}  + (1-\a)\n{\bz_{k+1}-\bw_k}^2\right]\notag \\ \leq \a &\n{\bz_k-\bz}^2 + (1-\a)\n{\bw_k-\bz}^2 +2\t
      \lr{F(\bz_{k})-F(\bw_{k-1}),\bz-\bz_{k}} \notag \\&+2\t \E_k\left[\lr{F_{\xi_k}(\bz_{k})-F_{\xi_k}(\bw_{k-1}),
          \bz_{k}-\bz_{k+1}} - \a\n{\bz_{k+1}-\bz_k}^2\right].
    \end{align}
 Using~\Cref{asmp: asmp1}(iv), Cauchy-Schwarz and Young's inequalities, we can bound the last
 line above as
\begin{align}\label{eq: cs_ineq}
&\E_k[2\t\langle F_{\xi_k}(\bz_k) - F_{\xi_k}(\bw_{k-1}), \bz_k
  - \bz_{k+1} \rangle -  \a\n{\bz_{k+1}-\bz_k}^2]\notag \\ \leq & \E_k
  \left[\frac{\t^2}{\a\gamma} \| F_{\xi_k}(\bz_k) - F_{\xi_k}(\bw_{k-1}) \|^2+\a\gamma \| \bz_{k+1} - \bz_k\|^2-  \a\n{\bz_{k+1}-\bz_k}^2\right] \notag \\
\leq \frac{(1-\a)\gamma}{L^2} &\E_k \| F_{\xi_k}(\bz_k) - F_{\xi_k}(\bw_{k-1}) \|^2 - (1-\gamma) \alpha \| \bz_{k+1} - \bz_k\|^2 \notag \\
&\leq  (1-\a)\gamma\n{\bz_k-\bw_{k-1}}^2 - (1-\gamma)\alpha \| \bz_{k+1} - \bz_k\|^2.
\end{align}
Adding~\eqref{eq: eg1_wexp} and~\eqref{eq: cs_ineq} to~\eqref{eq:
  forb_2},  we obtain
\begin{equation*}
    \E_k[ \Phi_{k+1}(\bz)] \leq \Phi_k(\bz) -
    (1-\alpha)(1-\gamma)\n{\bz_k-\bw_{k-1}}^2 - (1-\gamma)\alpha \| \bz_{k+1} - \bz_k\|^2.
\end{equation*}
The rest of the proof is the same as~\Cref{th: fbf}. The only difference is that instead of~\eqref{eq: fix_pt_fbf}, we have
\begin{equation}
\bz_{k+1}(\theta) - \bzb_k(\theta) + \tau \left( F_{\xi_k}(\bz_k(\theta)) - F_{\xi_k}(\bw_{k-1}(\theta)) \right) + \tau \left( F(\bz_{k+1}(\theta)) - F(\bw_k(\theta)) \right)\in \tau \left( F+G \right)(\bz_{k+1}(\theta)),
\end{equation}
which gives the same conclusion as $F_{\xi}$ is continuous for all $\xi$, $\bz_{k+1} - \bzb_k \to 0$, $\bz_{k+1} - \bw_k \to 0$ almost surely.
\end{proof}
\begin{remark}\label{rem: rem1_label}
Even though we will set the parameters $\alpha$, $p$, $\tau$ by
optimizing complexity, we observe that the requirements
in~\Cref{th: forb} allows step sizes arbitrary close to
$\frac{1}{2L}$. This already shows flexibility of the analysis,
compared to the strict requirement of $\tau=\frac{p}{4L}$
in~\cite{alacaoglu2020simple}.
\end{remark}
The improvement in the step size choice is due to using $\bar \bz_k$ which allows us to use tighter estimations whereas the analysis in~\cite{alacaoglu2020simple} needs to make use of multiple Young's inequalities.
In particular, we use $\bz_\ast$ as an anchor point in~\eqref{eq:
  eg1_wexp}, whereas~\cite{alacaoglu2020simple} uses $\bz_k$ as anchor
point, which requires Young's inequalities to transform to $\bz_{k-1}$
and obtain a telescoping sum.
Finally, as~\Cref{cor: fbf}, we give the complexity of the algorithm for solving VI in the spirit of~\Cref{sec: euc_sub_rate}.
\begin{corollary}\label[corollary]{cor: forb}
Let $\alpha = 1-p = 1 - \frac{2}{N}$ and $\bz^K = \frac{1}{K} \sum_{k=0}^{K-1} \bz_k$. Then, the total complexity to get an $\varepsilon$-accurate solution to~\eqref{eq: prob_vi} is $\mathcal{O}\left(N+\frac{\sqrt{N}L}{\epsilon} \right)$.
\end{corollary}

\subsection{Linear convergence}\label{sec: lin_conv}
In this section, we illustrate how to obtain linear convergence of Alg.~\ref{alg: eg1} for solving VI~\eqref{eq: prob_vi} when $g$ is $\mu$-strongly convex. Alternatively, one can replace this assumption with strong monotonicity of $F$, which we omit for brevity. One can use the same arguments for FBF and FoRB variants in the previous sections to show linear convergence for solving strongly monotone inclusions.
\begin{theorem}\label{th: th_lin}
Let~\Cref{asmp: asmp1} hold, $g$ be $\mu$-strongly convex, and
$\bz_*$ be the solution of~\eqref{eq: prob_vi}. If we set
$\a = 1-p$ and $\tau = \frac{\sqrt{p}}{2L}$ in Alg.~\ref{alg: eg1}, then
it holds that
\begin{equation*}
\mathbb{E}\| \bz_{k} - \bz_\ast \|^2  \leq \left( \frac{1}{1+c/3}\right)^k  \frac{2}{1-p}\| \bz_0 - \bz_\ast \|^2,
\end{equation*}
with $c=\min\left\{\frac{3p}{8}, \frac{\sqrt p \mu}{2L} \right\}$.
\end{theorem}
\begin{proof}
In~\eqref{eq: prox_ineqs}, we use strong convexity of $g$ to have an additional term $\frac{\tau\mu}{2} \|\bz_{k+1} - \bz \|^2$ on the right-hand side of the first inequality. Next, we continue as in the proof of~\Cref{lem: eg1} to obtain, instead of~\eqref{eq: eg1_main1},
\begin{multline*}
\left( 1+ \tau\mu \right)\mathbb{E}_k\left[\| \bz_{k+1} -\bz_\ast \|^2 \right] \leq \alpha \|\bz_k-\bz_\ast\|^2 + (1-\alpha)\|\bw_k -\bz_\ast\|^2 -(1-\alpha)(1-\gamma) \|\bz_{k+1/2} -\bw_k\|^2 \\
-(1-\gamma) \mathbb{E}_k\left[\| \bz_{k+1} - \bz_{k+1/2} \|^2\right].
\end{multline*}
We add~\eqref{eq: eg1_wexp} to this inequality after using
the tower property, to deduce
\begin{multline*}
\left(\a + \tau\mu \right)\mathbb{E}_k\left[\| \bz_{k+1} -\bz_\ast \|^2 \right] + \frac{1-\alpha}{p} \mathbb{E}_k \left[ \| \bw_{k+1} - \bz_\ast \|^2 \right] \leq \alpha \|\bz_k-\bz_\ast\|^2 + \frac{1-\alpha}{p} \|\bw_k -\bz_\ast\|^2 \\
-(1-\gamma)\Big((1-\a) \|\bz_{k+1/2} -\bw_k\|^2
+  \mathbb{E}_k\left[\| \bz_{k+1} - \bz_{k+1/2} \|^2\right] \Big).
\end{multline*}
Since we set $\a = 1-p$ and $\gamma
= \frac 1 2 $, we can rewrite it as
\begin{multline}
\left(1-p + \tau\mu \right)\mathbb{E}_k\left[\| \bz_{k+1} -\bz_\ast \|^2 \right] +  \mathbb{E}_k \left[ \| \bw_{k+1} - \bz_\ast \|^2 \right] \leq (1-p) \|\bz_k-\bz_\ast\|^2 +  \|\bw_k -\bz_\ast\|^2 \\
-\frac 1 2 \left(p \|\bz_{k+1/2} -\bw_k\|^2
+  \mathbb{E}_k\left[\| \bz_{k+1} - \bz_{k+1/2} \|^2\right] \right).\label{eq: eg1_lin1.5}
\end{multline}
Next, by $2\n{u}^2+2\n{v}^2\geq \n{u+v}^2$ applied two times,
\begin{align*}
\frac{2c}{3} \mathbb{E}_k\|\bz_{k+1} - \bz_\ast \|^2 &\geq \frac{c}{3} \mathbb{E}_k\| \bw_{k+1} - \bz_\ast \|^2 - \frac{2c}{3} \mathbb{E}_k\left[\mathbb{E}_{k+1/2}\| \bz_{k+1} - \bw_{k+1} \|^2\right] \\
&= \frac{c}{3} \mathbb{E}_k\| \bw_{k+1} - \bz_\ast \|^2 - \frac{2c(1-p)}{3} \mathbb{E}_k\| \bz_{k+1} - \bw_{k} \|^2 \\
&\geq \frac{c}{3} \mathbb{E}_k\| \bw_{k+1} - \bz_\ast \|^2- \frac{4c}{3} \mathbb{E}_k\| \bz_{k+1} - \bz_{k+1/2} \|^2 - \frac{4c}{3} \| \bz_{k+1/2} - \bw_k\|^2.
\end{align*}
Using this inequality in~\eqref{eq: eg1_lin1.5}  and that $c\leq
\frac{\sqrt p\mu}{2L} =\t \mu $ gives us
\begin{multline}
\left(1-p+ \frac{c}{3} \right)\mathbb{E}_k\left[\| \bz_{k+1} -\bz_\ast \|^2 \right] + \left(1+ \frac{c}{3} \right) \mathbb{E}_k \left[ \| \bw_{k+1} - \bz_\ast \|^2 \right] \leq (1-p) \|\bz_k-\bz_\ast\|^2 +  \|\bw_k -\bz_\ast\|^2 \\
-\frac 1 2\left(p \|\bz_{k+1/2} -\bw_k\|^2
+ \mathbb{E}_k \| \bz_{k+1} - \bz_{k+1/2} \|^2 \right) + \frac{4c}{3} \left( \| \bz_{k+1/2} - \bw_k \|^2 + \mathbb{E}_k\|\bz_{k+1} - \bz_{k+1/2} \|^2 \right).\label{eq: eg1_lin2}
\end{multline}
By our choice of $c$, we have $\frac{4c}{3} \leq \frac{p}{2}$ and,
therefore, the second line of~\eqref{eq: eg1_lin2} is
nonpositive. Using $1-p+\frac c 3 > (1-p)(1+\frac c3)$ and taking
total expectation, yields
\begin{equation*}
\left(1+\frac{c}{3}\right) \mathbb{E}\left[ (1-p) \|\bz_{k+1} - \bz_\ast \|^2 + \| \bw_{k+1} - \bz_\ast \|^2 \right] \leq \mathbb{E} \left[ (1-p) \|\bz_{k} - \bz_\ast \|^2 + \| \bw_{k} - \bz_\ast \|^2 \right].
\end{equation*}
By iterating this inequality, we obtain
\begin{equation*}
(1-p) \mathbb{E} \| \bz_{k} - \bz_\ast \|^2 \leq \left( \frac{1}{1+c/3} \right)^k (2-p) \| \bz_0 - \bz_\ast \|^2,
\end{equation*}
which gives the result.
\end{proof}
\begin{corollary}\label[corollary]{cor: lin_conv}
Let $p = \frac{2}{N}$, $\tau = \frac{\sqrt{p}}{2L}$. The total average complexity is $\mathcal{O}\left(\left(N + \frac{\sqrt{N}L}{\mu} \right)\log\frac{1}{\epsilon} \right)$.
\end{corollary}
\begin{proof}
The $\e$-accuracy is reached after
$\cO(\log\frac{1}{\e}/\log(1+\frac c 3))$ iterations. This yields a factor
$\frac{pN+2}{\log(1+\frac c 3)}\approx \frac{3}{c}(pN+2)$
in total complexity. Using our choice for $c$, we obtain total average  complexity
\begin{equation*}
\max\left\{ \frac{8}{p}, \frac{6L}{\sqrt p \mu} \right\}(pN+2) \leq \frac{32}{p} + \frac{24L}{\sqrt{p}\mu} = 16N + \frac{12\sqrt{2N}L}{\mu}.
\end{equation*}
We lastly multiply the last estimate with $\log\left( \varepsilon^{-1}\right)$.
\end{proof}

\begin{remark}
In this case, Alg.~\ref{alg: eg1} has complexity
$\mathcal{O} \left(\left(N + \frac{\sqrt{N}L}{\mu}\right)\log\frac{1}{\varepsilon}\right)$,
compared to the deterministic methods
$\mathcal{O} \left( \frac{NL_F}{\mu}\log\frac{1}{\varepsilon}\right)$.
This complexity recovers the previously obtained result in~\cite{balamurugan2016stochastic} and~\cite[Section 5.4]{carmon2019variance}, where our advantage is having algorithmic parameters independent of $\mu$ and having more general assumptions.
\end{remark}
\section{Applications}\label{sec: apps}
\subsection{Bilinear min-max problems}\label{sec: bilinear}
In this section, we analyze the overall complexity of our method compared to deterministic extragradient and show the complexity improvements.

\paragraph{Notation.} For a vector $\bx$ we use $x_i$ to denote its
$i$-th coordinate and for an indexed vector $\bx_k$ it is
$x_{k,i}$. For a matrix $A\in\mathbb{R}^{m\times n}$ we denote a number of its non-zero entries
by $\nnz(A)$; it is exactly the complexity of computing $A\bx$ or
$A^\top \by$. We use the spectral, Frobenius and max norms of $A$
defined as $\n{A} = \s_{\max}(A)$,
$\n{A}_\Fr = \sqrt{\sum_{i,j} A_{ij}^2}=\sqrt{\sum_{i=1}^{\rank(A)}
  \s_i(A)^2}$, and $\n{A}_{\max}=\max_{i,j}|A_{ij}|$. For $i$-th row
and $j$-th column of $A$ we use a convenient notation $A_{i:}$ and
$A_{:j}$.
Here, for simplicity, we measure complexity in terms of arithmetic operations.
\paragraph{Problem.} The general problem that we consider is
\begin{equation*}
\min_{\bx\in\mathbb{R}^n}\max_{\by\in\mathbb{R}^m} \langle A\bx, \by\rangle + g_1(\bx) - g_2(\by),
\end{equation*}
where $g_1, g_2$ are proper convex lsc functions.
We can formulate this problem as a VI by setting
\begin{equation}\label{eq:F and g}
F(\bz) = F(\bx,\by) = \binom{A^\tr \by}{-A\bx}, \qquad g(\bz) = g_1(\bx)
+ g_2(\by).
\end{equation}

\subsubsection{Linearly constrained minimization}\label[section]{sec: lin_cons}
A classical example of bilinear saddle point problems is linearly constrained minimization
\begin{equation*}
\min_{\bx\in\mathbb{R}^n} f(\bx): A\bx = b,
\end{equation*}
where $f$ is proper convex lsc.
The equivalent min-max formulation corresponds to~\eqref{eq:F and g} when $g_1(\bx) = f(\bx)$ and $g_2(\by) = \langle b,\by \rangle$.

We will instantiate Alg.~\ref{alg: eg1} for this problem.
To make our presentation clearer, we consider only the most common
scenario when $\nnz(A) > m + n$. In this setting, deterministic methods
(extragradient, FBF, FoRB, etc.) solve \eqref{eq:minmax} with
$\mathcal{O}\left(\nnz(A)\n{A}\e^{-1}\right)$ total
complexity. As we see in the sequel, variance reduced methods provide us
$\mathcal{O}\left(\nnz(A) + \sqrt{\nnz(A)(m+n)}\n{A}_\Fr\e^{-1}\right)$ total
complexity.
We now describe the definition of $F_\xi$ with two oracle choices.
The first choice is the version of ``importance'' sampling described in~\Cref{sec: prelim}.

\paragraph{Oracle 1.} The fixed distribution (the same in every iteration) is defined as
\[F_\xi(\bz) = \binom{\frac{1}{r_i} A_{i:}y_i}{-\frac{1}{c_j}
    A_{:j}x_j},\quad \qquad  \Pr\{\xi=(i,j)\} = r_ic_j, \quad r_i
=\frac{\n{A_{i:}}_2^2}{\n{A}_{{\Fr}}^2},\quad c_j =
\frac{\n{A_{:j}}_2^2}{\n{A}_{{\Fr}}^2}. \]
In the view of~\Cref{asmp: asmp1}, the Lipschitz constant of $F_\xi$ can be computed as
\begin{align}
\E\n{F_{\xi}(\bz)}^2_2 &=\mathop{\E}_{i\sim r}\left[\frac{1}{r_i^2}\n{A_{i:}y_i}^2_2\right] +
  \mathop{\E}_{j\sim c}\left[\frac{1}{c_j^2}\n{A_{:j}x_j}^2_2\right]= \sum_{i=1}^m
  \frac{1}{r_i}\n{A_{i:}y_i}^2_{2}+\sum_{j=1}^n\frac{1}{c_j} \n{A_{:j}x_j}^2_{2}\notag \\ &= \sum_{i=1}^m\frac{1}{r_i}\n{A_{i:}}_{2}^2\left( y_i\right)^2+\sum_{j=1}^n \frac{1}{c_j}
  \n{A_{:j}}^2_{2}\left( x_j\right)^2= \n{A}_{\Fr}^2\n{\bz}^2_2.\label{eq: emr4}
\end{align}
\paragraph{Oracle 2.} The second stochastic oracle is slightly more
complicated, since it is iteration-dependent as~\cite{carmon2019variance}.
We use the setting of \Cref{asmp: asmp2}. Given $\bu = (\bu^x,\bu^y)$ and $\bv = (\bv^x, \bv^y)$, for
$\bz=(\bx,\by)$, we
define
\[F_\xi(\bz) = \binom{\frac{1}{r_i} A_{i:}y_i}{-\frac{1}{c_j}
    A_{:j}x_j}, \qquad \Pr\{\xi=(i,j)\} = r_ic_j, \quad r_i
  =\frac{|u^y_i - v^y_i|^2}{\n{\bu^y-\bv^y}^2},\quad c_j =
  \frac{|u^x_j-v^x_j|^2}{\n{\bu^x - \bv^x}^2}, \] and call the
described distribution as $Q(\bu,\bv)$.
Similarly, in every iteration of Alg.~\ref{alg: mp1}  we define a distribution
$Q(\bz_{k+1/2}^s, \bw^s)$ and sample $\xi$ according to it.

 Clearly, as before, $F_{\xi}$
is unbiased.
It is easy to show that this oracle is variable $\n{A}_{\Fr}$-Lipschitz.
Its proof is similar to the variable Lipschitz derivation that we will include for matrix games with Bregman distances, in~\Cref{sec: matrix_games}.

\paragraph{Complexity.}
We suppose that computing proximal operators $\prox_{g_1}$,
$\prox_{g_2}$ can be done efficiently in $\tilde \cO(m+n)$
complexity.
Our result in~\Cref{eq: eg1_th2} stated that Alg.~\ref{alg: eg1} has the rate $\mathcal{O}\left( \frac{L}{\sqrt{p}K} \right)$.
Given that the expected cost of each iteration is
$\mathcal{O}\left(p\nnz(A) + m+n\right)$, setting $p
= \frac{m+n}{\nnz(A)}$ gives us the average total complexity
\begin{equation}\label{eq: comp_lin_const}
\tilde{\mathcal{O}}\left( \nnz(A) + \frac{\sqrt{\mathrm{nnz}(A)(m+n)}\n{A}_\Fr}{\e} \right).
\end{equation}
It is easy to see
  that Alg.~\ref{alg: mp1} has the same complexity if we set  $K=\left\lceil\frac{\nnz(A)}{m+n}\right\rceil$. %

Compared to the complexity of deterministic methods, this complexity improves depending on the relation between $\|A\|_{\Fr}$ and $\|A\|$.
In particular, when $A$ is a square dense matrix, due to $\| A\|_\Fr
\leq \sqrt{\rank(A)}\|A\|$, the bound in~\eqref{eq: comp_lin_const} improves that of
deterministic VI methods.
In~\eqref{eq: comp_lin_const} we suppress $\| \bz_0 - \bz_*\|^2$ that is common to all methods considered in this paragraph.

Finally, we remark that the analysis in~\cite[Section 5.2]{carmon2019variance}
requires the additional assumption that $\bz\mapsto\langle F(\bz) + \tilde{\nabla}
f(\bz), \bz - \bu \rangle$ is convex for all $\bu$ to apply
to this case, where we  denote a subgradient of $f$ by $\tilde{\nabla} f$.
This assumption requires more structure on $f$.

\subsubsection{Matrix games}\label[section]{sec: matrix_games}
The problem in this case is written as
\begin{equation}\label{eq:minmax}
\min_{\bx\in\mathcal{X}}\max_{\by\in\mathcal{Y}} \langle A\bx, \by \rangle,
\end{equation}
where $A\in \R^{m\times n}$ and $\mathcal{X}\subset \R^n$,
$\mathcal{Y}\subset \R^m$ are
closed convex sets, projection onto each are easy to compute.
In view of~\eqref{eq:F and g}, we have $g(\bz) = \d_{\cX}(\bx)
+ \d_{\cY}(\by)$.
As we shall see, our complexities in this case recover the ones in~\cite{carmon2019variance}.
We refer to~\Cref{sec: rel_work} for a detailed comparison.

In the Euclidean setup, we suppose that the underlying space $\cZ = \R^n\times \R^m$ has a
Euclidean structure with the norm $\n{\cdot}_2$ and, hence, it coincides with
the dual $\cZ^*$.
In this case, we can use Oracle 1 and Oracle 2 from~\Cref{sec: lin_cons} and we obtain the same complexity as~\eqref{eq: comp_lin_const}.
The same discussions as~\Cref{sec: lin_cons} apply.
\begin{center}
  \underline{\textsc{Bregman setup}}
\end{center}
Let
$\cX = \Delta^n=\{\bx \in \R^n\colon \sum_{i=1}^nx_i=1, x_i\geq 0\}$
and $\cY=\Delta^m$. With this, problem \eqref{eq:minmax} is known as a
zero sum game.  In this case, deterministic algorithms formulated with a
specific Bregman distance (given below) have
$\mathcal{O}\left(\nnz(A)\n{A}_{\max}\e^{-1} \right)$ total
complexity. These settings are standard and we recall them only for
reader's convenience.

For $\cZ=\R^{m+n}$ and $\bz = (\bx,\by) \in \cZ$ we define
$\n{\bz}=\sqrt{\n{\bx}_1^2+\n{\by}_1^2}$. Correspondingly,
$\cZ^*=\left(\R^{m+n},\n{\cdot}_{*}\right)$ is the dual space with
$\n{\bz^*} = \sqrt{\n{\bx^*}^2_{\infty}+\n{\by^*}^2_{\infty}}$ for
$\bz^*=(\bx^*,\by^*)$. For $\bz=(\bx,\by)\in \Delta^n\times \Delta^m$
we use the negative entropy $h_1(\bx) = \sum_{i=1}^{n}x_i \log
x_i$, $h_2(\by) = \sum_{i=1}^{m}y_i \log
y_i$ and set $h(\bz) = h_1(\bx)+h_2(\by) = \sum_{i=1}^{m+n}z_i \log
z_i$. Then we define the Bregman distance  as
\[D(\bz, \bz') = h(\bz) - h(\bz') - \lr{\nabla h(\bz'),\bz-\bz'}=\sum_i
  z_i\log\frac{z_i}{z_i'}. \]
Of course, this definition requires $\bz'$ to be in the relative
interior of $\Delta^n\times \Delta^m$; normally it is satisfied
automatically for the iterates of the algorithm (including our
Alg.~\ref{alg: mp1}).

If we choose $\bz_0 =(\bx_0,\by_0)$ with $\bx_0 =\frac 1 n \one_n$,
$\by_0=\frac{1}{m}\one_m$, it is easy to see that
\[\max_{\bz \in \Delta^n\times \Delta^m}D(\bz,\bz_0)\leq \log n + \log
  m = \log(mn).\]
We know that $D$ satisfies $D(\bz,\bz')\geq \frac 1 2\n{\bz-\bz'}^2$
for all $\bz,\bz'\in \Delta^n\times \Delta^m$.
Deterministic algorithms have constant $\n{A}_{\max}$ in their
complexity, since $F$ defined in \eqref{eq:F and g} is
$\n{A}_{\max}$-Lipschitz:
\[\n{F(\bz)}_*^2 = \n{A^\top \by}_{\infty}^2 + \n{A\bx}^2_{\infty}\leq
  \n{A}^2_{\max}(\n{\bx}_1^2 + \n{\by}_1^2) =   \n{A}^2_{\max}\n{\bz}^2.  \]

\paragraph{Oracle.} The stochastic oracle here is similar to the Oracle 2 in~\Cref{sec: lin_cons} for the
Euclidean case, but with adjustment to the $\ell_1$-norm.  Again we
are in the setting of \Cref{asmp: asmp2}. Given $\bu = (\bu^x,\bu^y)$ and
$\bv = (\bv^x, \bv^y)$, for $\bz=(\bx,\by)$, we define
\[F_\xi(\bz) = \binom{\frac{1}{r_i} A_{i:}y_i}{-\frac{1}{c_j}
    A_{:j}x_j}, \qquad \Pr\{\xi=(i,j)\} = r_ic_j, \quad r_i
  =\frac{|u^y_i - v^y_i|}{\n{\bu^y-\bv^y}_1},\quad c_j =
  \frac{|u^x_j-v^x_j|}{\n{\bu^x - \bv^x}_1}, \] and call the
described distribution as $Q(\bu,\bv)$. We show that $F_\xi$ is
variable $\n{A}_{\max}$-Lipschitz in view of~\Cref{def: def1}. Indeed, we have
\begin{align*}
\mathop{\mathbb{E}}_{\xi \sim Q(\bu,\bv)}[ \| F_{\xi}(\bu)-
  F_{\xi}&(\bv) \|^2_*]=\mathop{\mathbb{E}}_{\xi \sim
  Q(\bu,\bv)}\left[ \| F_{\xi}(\bu -\bv) \|^2_{*}\right] \\ &=
\mathop{\mathbb{E}}_{i\sim r} \left[ \frac{1}{r_i^2} \|
  A_{i:}(u^y_i-v^y_i)\|^2_{\max} \right] + \mathop{\mathbb{E}}_{j\sim c}
\left[ \frac{1}{c_j^2} \| A_{:j} (u^x_j-v^x_j)\|^2_{\max} \right] \\
 &= \sum_{i=1}^m \frac{1}{r_i} \| A_{i:}\|^2_{\max} |u^y_i-v^y_i|^2 +
 \sum_{j=1}^n \frac{1}{c_j} \| A_{:j}\|^2_{\max}|u^x_j-v^x_j|^2\\&\leq
 \sum_{i=1}^m  \| A\|^2_{\max} |u^y_i-v^y_i|\n{\bu^y-\bv^y}_1 +
 \sum_{j=1}^n \|A\|^2_{\max}|u^x_j-v^x_j|\n{\bu^x-\bv^x}_1\\ &=\n{A}^2_{\max}\left(\|\bu^y-\bv^y \|^2_1+ \|\bu^y-\bv^x \|^2_1\right) =\|A\|^2_{\max} \| \bu-\bv \|^2.
\end{align*}
Similarly, in every iteration of Alg.~\ref{alg: mp1} we define a
distribution $Q(\bz_{k+1/2}^s, \bw^s)$ and sample $\xi_k^s$ according
to it.  This stochastic oracle was already used in
\cite{grigoriadis1995sublinear} and used extensively after that, see
\cite{nesterov22first,clarkson2012sublinear} and references
therein. In \cite{carmon2019variance} this oracle was called
``sampling from the difference''.
\paragraph{Complexity.}
In this case, the complexity of deterministic algorithms (Mirror Prox, FoRB) is \\$\mathcal{O}\left( \nnz(A) \| A\|_{\max}\varepsilon^{-1} \right)$.
Our result in~\Cref{cor:compl_br} stated that Alg.~\ref{alg: mp1} has the rate $\mathcal{O}\left( \frac{L}{\sqrt{K}S} \right)$.
Given that the cost of each epoch of Alg.~\ref{alg: mp1} is
$\mathcal{O}\left(\nnz(A) + K(m+n)\right)$, setting $K
= \left\lceil\frac{\nnz(A)}{m+n}\right\rceil$ gives us the total complexity
\begin{equation*}
\tilde{\mathcal{O}}\left(\nnz(A)+ \frac{\sqrt{\nnz(A) (m+n)}\|A\|_{\max}}{\varepsilon} \right),
\end{equation*}
which, in the square dense case, improves the deterministic complexity by $\sqrt{n}$.

\paragraph{Updates.} For concreteness we specify updates in lines~\ref{mp:l1}--\ref{mp:l2}
of Alg.~\ref{alg: mp1}. Let
$\bw^s = (\bu, \bv)$, $\bwb^s=(\bub^s, \bvb^s)$.
\begin{align*}
 \nabla h_1 (\bx_{k+1/2}^s) &= \a \nabla h_1(\bx_k^s) + (1-\a)\nabla
 h_1(\bub^s) - \t A^\top \bv^s\\
 \nabla h_2 (\by_{k+1/2}^s) &= \a \nabla h_2(\by_k^s) + (1-\a)\nabla h_2(\bvb^s) + \t A \bu^s
\end{align*}
Then we form a distribution $Q(\bz_{k+1/2}^s, \bw^s)$
\[\Pr\{\xi=(i,j)\} = r_ic_j, \quad r_i
  =\frac{|y_{k+1/2,i}^s - v_i^s|}{\n{\by_{k+1/2}^s-\bv^s}_1},\quad c_j
  = \frac{|x_{k+1/2,i}^s - u_i^s|}{\n{\bx_{k+1/2}^s-\bu^s}_1} \]
and sample $\xi_k = (i,j)$ according to $Q(\bz_{k+1/2}^s, \bw^s)$. Finally, we update
$\bx_{k+1}^s$ and $\by_{k+1}^s$ as
\begin{align*}
 \nabla h_1 (\bx_{k+1}^s) &= \a \nabla h_1(\bx_k^s) + (1-\a)\nabla
 h_1(\bub^s) - \t A^\top \bv^s -
 \frac{\t}{r_{i}}A_{i:}(y_{k+1/2,i}^s-v_i^s)\\
 &=\nabla h_1 (\bx_{k+1/2}^s) -
 \t A_{i:}\n{\by_{k+1/2}^s-\bv^s}\sign(y_{k+1/2,i}^s-v_i^s)\\[1mm]
 \nabla h_2 (\by_{k+1}^s) &=\nabla h_2 (\by_{k+1/2}^s) + \t A_{:j}\|\bx_{k+1/2}^s - \bu^s \|\sign(x_{k+1/2,j}^s-u_j^s)
\end{align*}

Switching from dual variables $\nabla h_1(\bx)$ to primal $\bx$ is
elementary by  duality:
\[X = \nabla h_1(\bx)\quad \iff \quad \bx = \nabla h^*_1(X) =
  \frac{(e^{X_1},\dots,e^{X_n})}{\sum_{i=1}^n e^{X_i}}\] and similarly
for $\by$. Updates for $\bw$ and $\nabla h(\bwb)$ are
straightforward by means of incremental averaging.

\subsection{Nonbilinear min-max problems}
An important example of nonbilinear min-max problems is constrained optimization
\begin{align*}
&\min_{\bx\in\mathcal{X}} f(\bx) \quad \text{ subject to }\quad h_i(\bx) \leq 0, \text{ for } i \in [N],
\end{align*}
where $f, h_i$ are smooth convex functions.
We can map this problem
to the VI template~\eqref{eq: prob_vi} by setting
\begin{equation*}
  F = \binom{\nabla f(\bx) + \sum_{i=1}^N y_i \nabla
    h_i(\bx)}{-\bigl(h_1(\bx),\dots h_N(\bx)\bigr)^\top}, \qquad g(\bz) = \delta_{\mathcal{X}}(\bx) + \delta_{\mathbb{R}_+^N} (\by).
\end{equation*}
One possible choice for stochastic oracles is to set
\begin{equation}\label{eq: stoc_oracle_const}
F_i(\bz) = \binom{\nabla f(\bx) + N y_i \nabla h_i(\bx)}{Nh_i(\bx)\mathbf{e}_i},
\end{equation}
where $\mathbf{e}_i$ is the $i$-th standard basis vector.
Of course, this form of the oracle will not necessarily be a good choice for specific applications.

In particular, as discussed in~\Cref{sec: outline_results} and in the corollaries of our main theorems, our results will apply in their full generality and they will improve deterministic complexity as long as $L\leq \sqrt{N}L_F$, where $L$ is the Lipschitz constant corresponding to stochastic oracle in view of~\Cref{asmp: asmp1} and $L_F$ is for the full operator.
However, it is not clear that the generic choice in~\eqref{eq: stoc_oracle_const} will satisfy this requirement.
Therefore, one should be careful to design suitable oracles depending on the particular structure of the problem to ensure complexity improvements.
We refer to~\Cref{sec: rel_work} for a detailed comparison with related works.

\begin{figure*}[t]
\captionsetup{font=small}
\begin{center}
  \includegraphics[width=0.32\columnwidth]{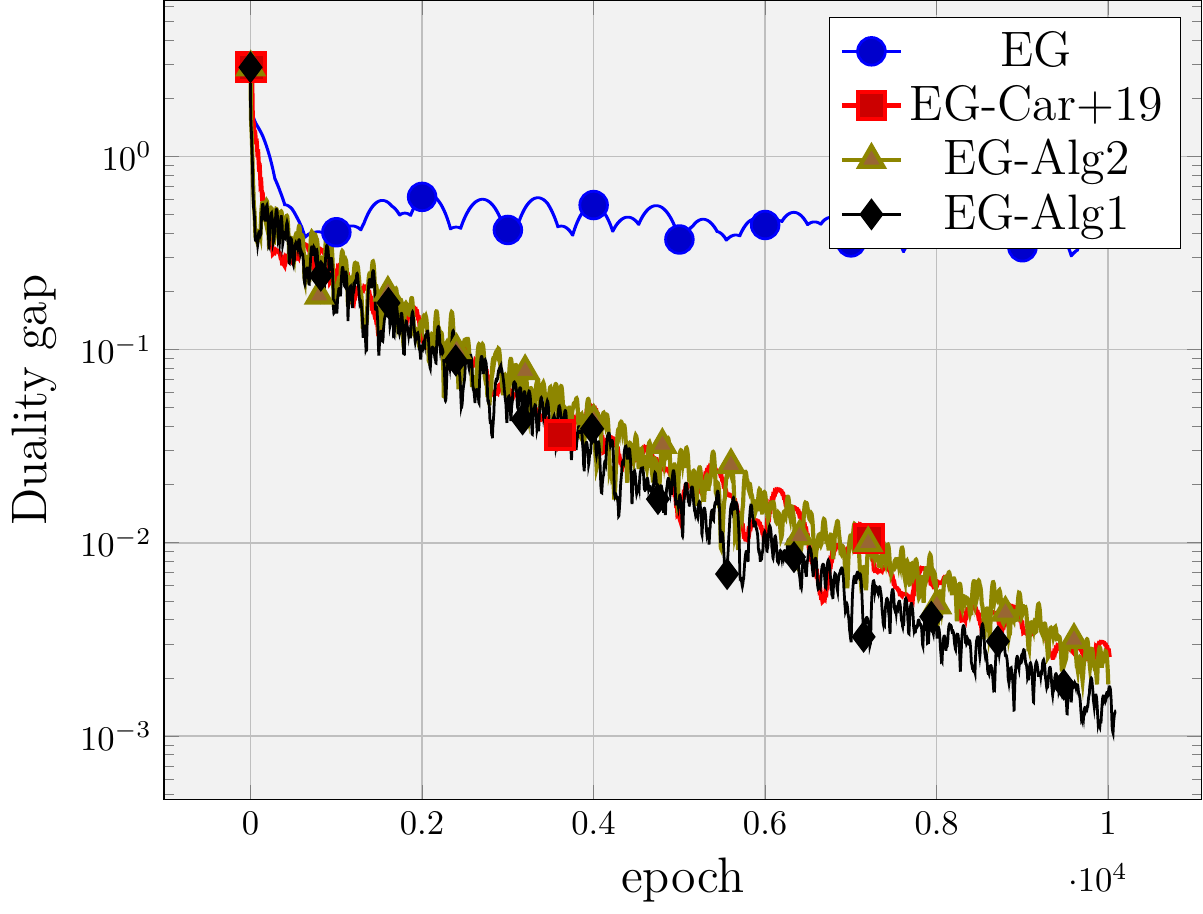}
\includegraphics[width=0.32\columnwidth]{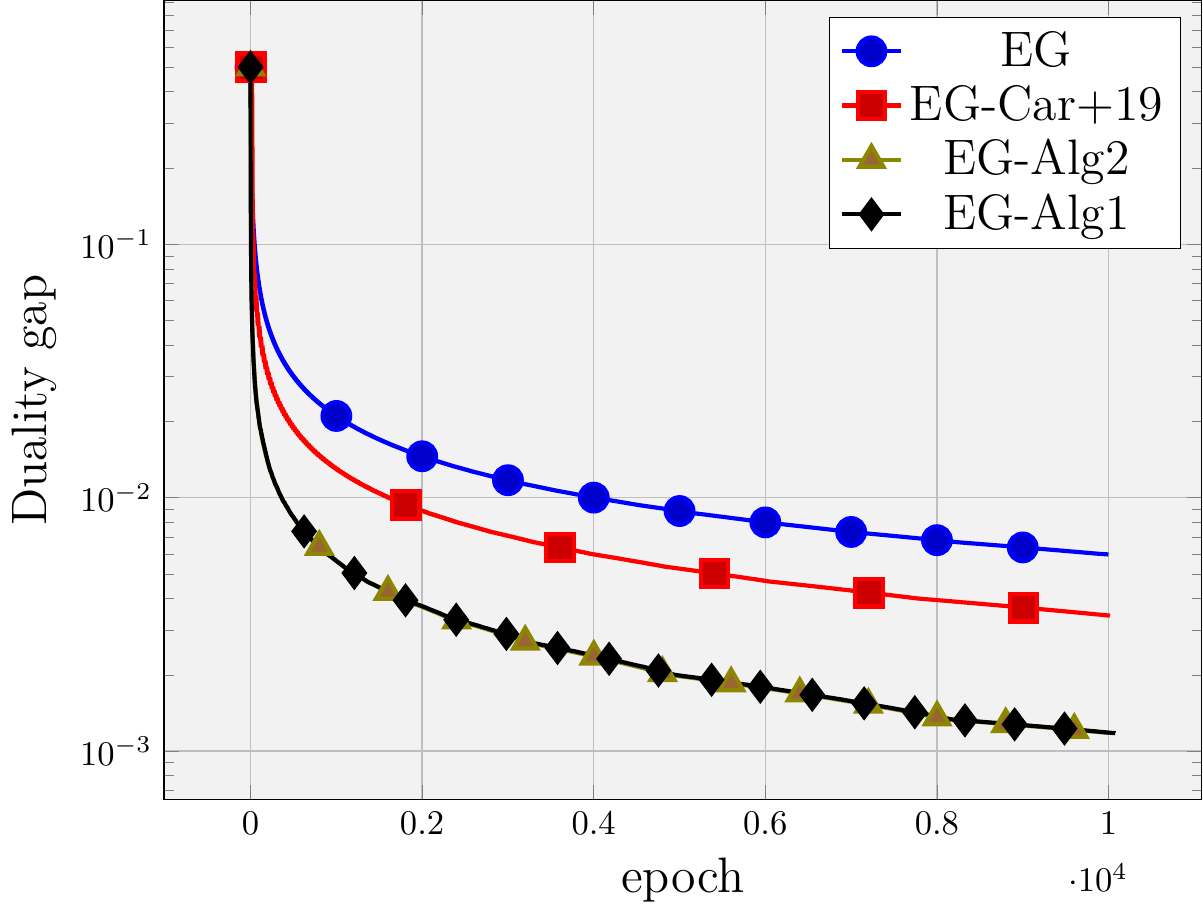}
\includegraphics[width=0.32\columnwidth]{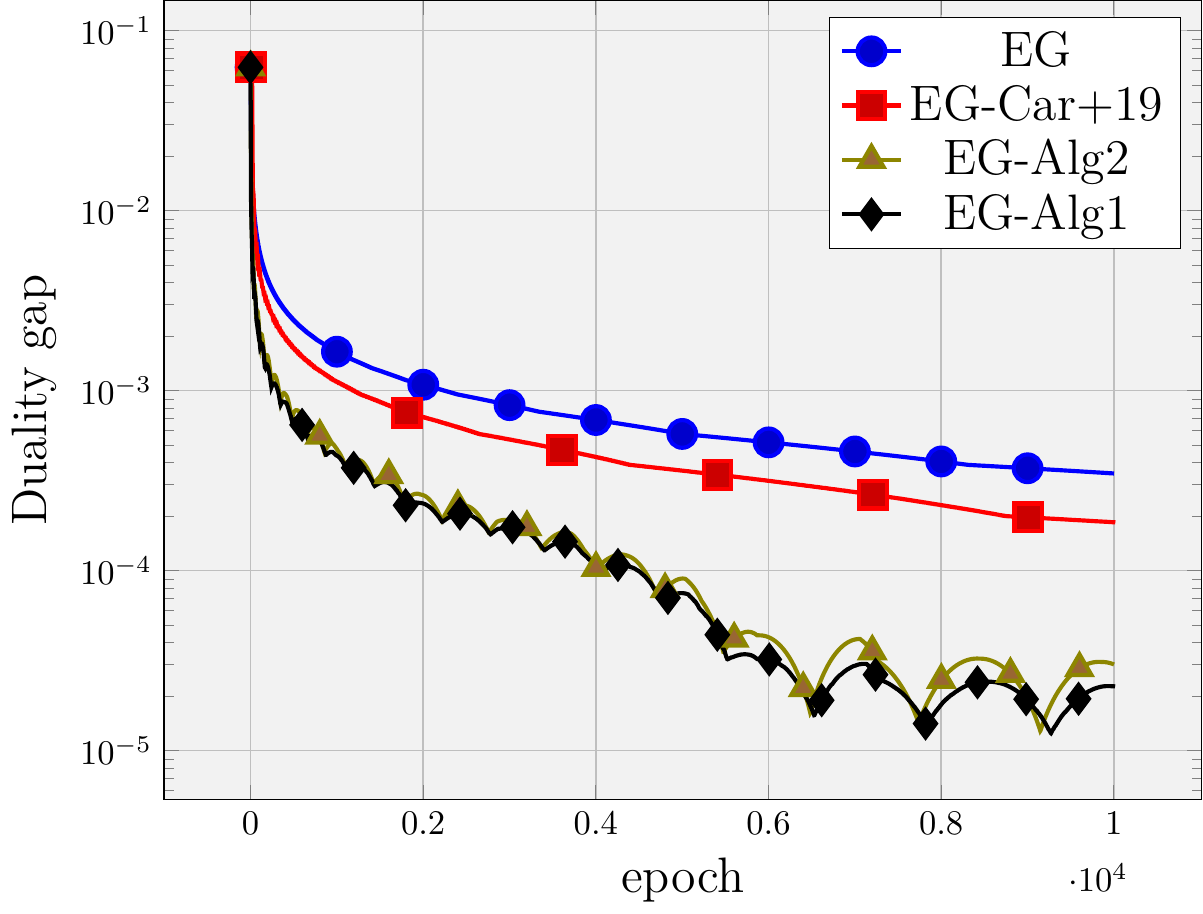}
\caption{Euclidean setup. left: policeman and burglar matrix~\cite{nemirovski2013mini}, middle, right: two test matrices given in~\cite[Section 4.5]{nemirovski2009robust}.}
\label{fig:1}
\end{center}
\vskip -0.3cm
\end{figure*}

\begin{figure*}[t]
\captionsetup{font=small}
\begin{center}
\includegraphics[width=0.32\columnwidth]{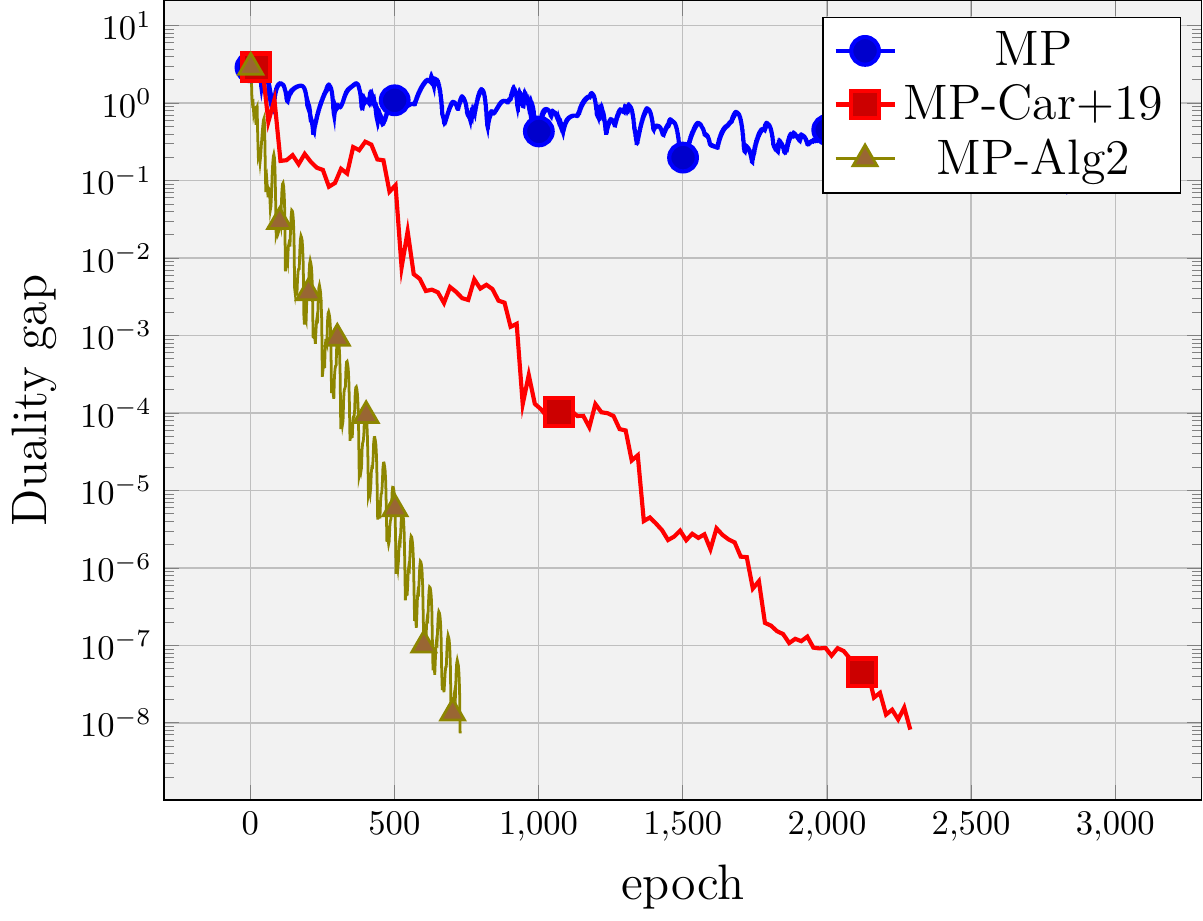}
\includegraphics[width=0.32\columnwidth]{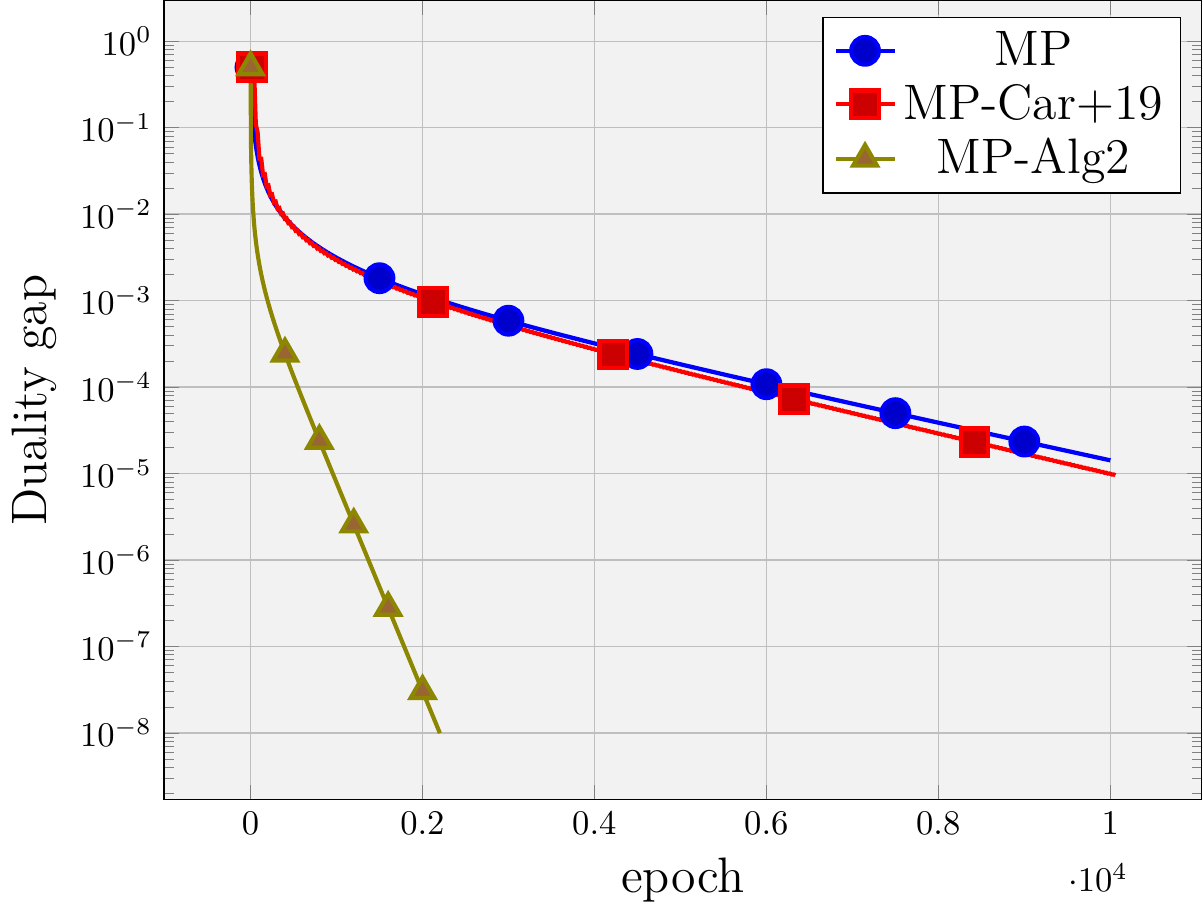}
\includegraphics[width=0.32\columnwidth]{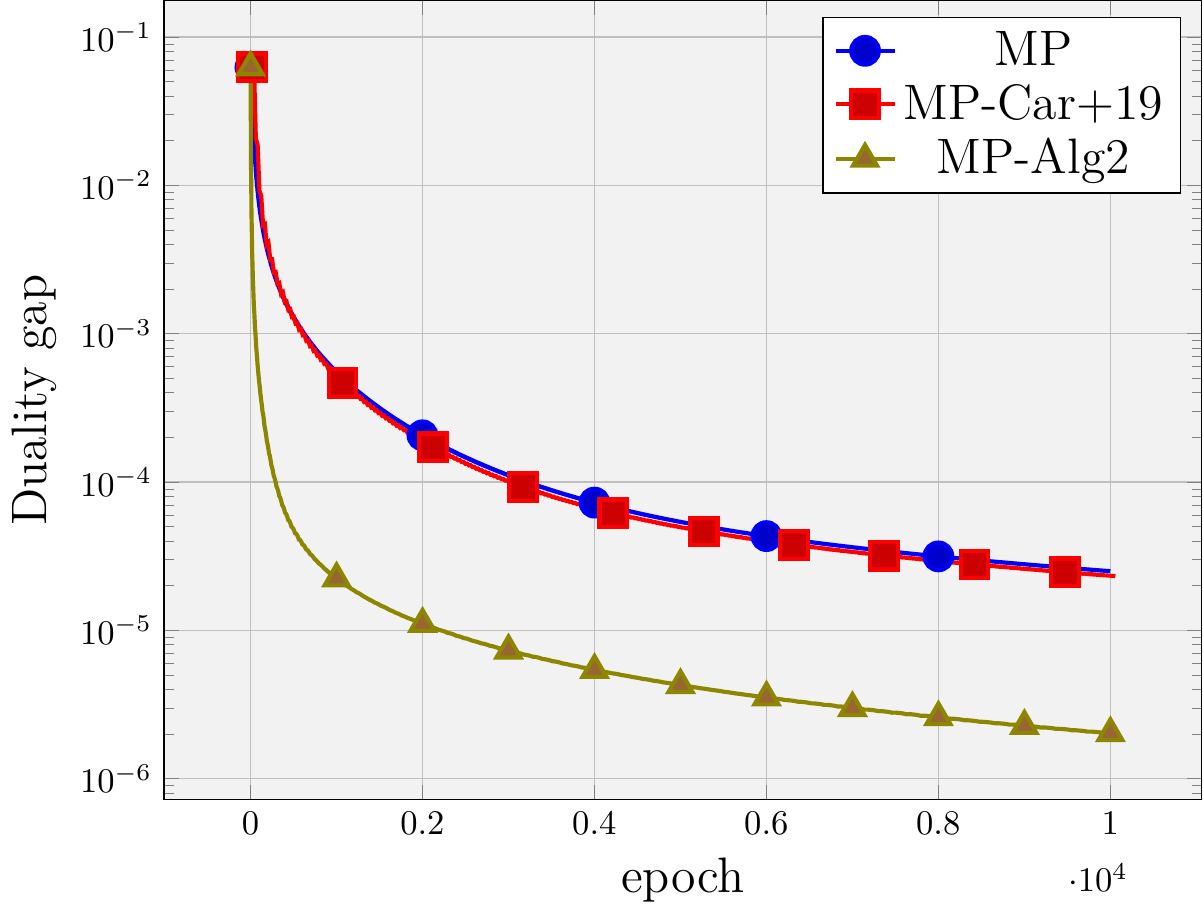}
\caption{Entropic setup. The same matrices in~\Cref{fig:1} used in the same arrangement.}
\label{fig:2}
\end{center}
\vskip -0.3cm
\end{figure*}
\section{Numerical experiments}\label{sec: numexp}
In this section, we  provide preliminary empirical evidence\footnote{Code can be
  found in \href{https://github.com/ymalitsky/VR_for_VI}{https://github.com/ymalitsky/vr\_for\_vi}}  on how variance
reduced methods for VIs perform in practice. By no means, this report
is exhaustive, but only an illustration for showing \textit{(i)} variance reduction helps in practice compared to deterministic methods and $\textit{(ii)}$ our approach is not only more general in theory but also offers practical advantages compared to the previous approach in~\cite{carmon2019variance}.

We focus on matrix games with simplex constraints in the Euclidean and entropic setups. In the Euclidean step, we use the projection to simplex from~\cite{condat2016fast}.
We
compare deterministic extragradient (EG), existing variance-reduced
method~\cite{carmon2019variance} (EG-Car+19) and proposed Alg.~\ref{alg:
  eg1} and Alg.~\ref{alg: mp1}. To distinguish from the Euclidean case, we
write `MP' instead of `EG' for all algorithms.  We have chosen three
test problems used in the literature~\cite{nemirovski2013mini,nemirovski2009robust} and fixed $m=n=500$.

For all problems, we use the largest step sizes allowed by theory.
In particular, EG uses $1/L_F$, where $L_F$ is the Lipschitz constant of the overall operator $F$.
We also use the reported parameters from~\cite{carmon2019variance} for EG-Car+19.
In the Euclidean case, by tracing the proof of~\cite[Proposition 2]{carmon2019variance}, we observed that one can improve the step size from $\eta = \frac{\alpha}{10L^2}$ to $\eta = \frac{\alpha}{4L^2}$, where $\alpha$ is defined to be $\frac{L\sqrt{m+n}}{\sqrt{\nnz(A)}}$ therein.
Therefore, we use the improved step size for EG-Car+19 for experiments with Euclidean setup.
However, in the Bregman setup, we did not find a way to improve the step size of EG-Car+19, so we use the reported one.

In our methods, we use the parameters from Remarks~\ref{rem: eg} and~\ref{rem: mp}.
For performance measure, we use duality gap, which can be simply computed as $\max_{i} (A\bx)_i - \min_j (A^\top \by)_j$ due to simplex constraints.
Cost of computing one $F$ is counted as an epoch, and the cost of stochastic oracles are counted accordingly to match the overall cost.

We report the results in Figures~\ref{fig:1} and~\ref{fig:2}.
We see that variance reduced variants consistently outperform deterministic EG in all cases, as predicted in theory.
Within variance reduced methods, due to the small step sizes of EG-Car+19, except the first dataset in the Euclidean setup, we observe our algorithms to also outperform EG-Car+19.
Especially in the Bregman setting, the difference is noticeable since the analysis of EG-Car+19 requires smaller step sizes.

\section{Conclusions}
We conclude by discussing a few potential directions that our results
could pave the way for.\\[2mm]
\textbf{Sparsity.}  An important consideration in practice is to adapt to sparsity of the data.  The recent work by~\cite{carmon2020coordinate} built on the algorithm in~\cite{carmon2019variance} and improved the complexity for matrix games in Euclidean setup, for sparse data, by using specialized data structures.  We suspect that these techniques can also be used in our
algorithms.\\[2mm]
\textbf{Stochastic oracles.}  As we have seen for bilinear and nonbilinear problems, harnessing the structure is very important for devising suitable stochastic oracles with small Lipschitz constants.
On top of our algorithms, an interesting direction is to study important nonbilinear min-max problems and devise particular Bregman distances and stochastic oracles to obtain complexity improvements.\\[2mm]
\textbf{New algorithms.}  For brevity, we only showed the application of our techniques for extragradient, FBF, and FoRB methods.  However, for more structured problems other extensions might be more suitable. Such structured problems arise, for example, when only partial strong convexity is present or when $F$ is the sum of a skew-symmetric matrix and a gradient of a convex function.

\section{Appendix}\label{sec: appendix}
\begin{proof}[Proof of~\Cref{th: eg1_th1}]
By the proof of~\Cref{lem: eg1}, without removing the term~$- \alpha
\| \bz_{k+1/2} - \bz_k\|^2$ in~\eqref{eq: eg1_ip2}, we have
\begin{multline}\label{eq: as_eg_1}
\mathbb{E}_k \left[ \Phi_{k+1}(\bz_\ast) \right] \leq \Phi_k(\bz_\ast) - (1-\gamma) \left((1-\alpha) \| \bz_{k+1/2} - \bw_{k} \|^2 + \mathbb{E}_k\left[\| \bz_{k+1} - \bz_{k+1/2} \|^2\right]\right) \\
- \alpha \| \bz_{k+1/2} - \bz_k\|^2.
\end{multline}
By Robbins-Siegmund theorem~\cite[Theorem 1]{robbins1971convergence},
we have that  $\Phi_k(\bz_\star)$ converges a.s.\
 and $\| \bz_{k+1/2} - \bz_k \|$, $\|\bz_{k+1/2}-\bw_k \|$ converges to $0$ a.s.

Let $\mathsf{Z_k} = \begin{bmatrix}\bz_k\\ \bw_k\end{bmatrix}$ and
$\mathsf{Z}_\ast = \begin{bmatrix} \bz_\ast \\ \bz_\ast \end{bmatrix}$,
then $\Phi_k(\bz_\ast) =\alpha\|\bz_k - \bz_\ast\|^2 + \frac{1-\alpha}{p}\|\bw_k - \bz_\ast \|^2 = \| \mathsf{Z}_k - \mathsf{Z}_\ast \|^2_Q$
with
$Q =
\mathrm{diag}(\alpha,\dots,\alpha,\frac{1-\alpha}{p},\dots,\frac{1-\alpha}{p})$.
Then applying \cite[Proposition~2.3]{combettes2015stochastic} to the
inequality $\mathbb{E}_k\| \mathsf{Z}_{k+1} - \mathsf{Z}_\ast \|^2_Q\leq \|
\mathsf{Z}_{k} - \mathsf{Z}_\ast \|^2_Q$, we can construct $\Xi$, with
$\mathbb{P}(\Xi) = 1$, such that for all $ \theta\in\Xi$ and $\forall \bz_\ast \in \Sol$
$\| \mathsf{Z}_k(\theta) - \mathsf{Z}_\ast \|_Q$ converges and
therefore, there exists $\Xi$ with $\mathbb{P}(\Xi)=1$, such that
\begin{equation}\label{eq: globalization_eq}
\text{$\forall \theta\in\Xi$ and $\forall \bz_\ast\in\Sol$,\quad
$ \alpha \| \bz_k(\theta) - \bz_\ast \|^2 +
\frac{1-\alpha}{p}\|\bw_k(\theta)-\bz_\ast \|^2$ converges.}
\end{equation}
Moreover, by taking total expectation on~\eqref{eq: as_eg_1}, we get $\sum_{k=1}^\infty \mathbb{E} \| \bz_{k+1} - \bz_{k+1/2} \|^2 < \infty$. By Fubini-Tonelli theorem, we have $\mathbb{E}\left[\sum_{k=1}^\infty \| \bz_{k+1} - \bz_{k+1/2} \|^2\right] < \infty$ and since $\sum_{k=1}^{\infty}\| \bz_{k+1} - \bz_{k+1/2} \|^2$ is nonnegative, $\sum_{k=1}^\infty \| \bz_{k+1} - \bz_{k+1/2} \|^2 < \infty$ a.s.\ and thus $\| \bz_{k+1} - \bz_{k+1/2} \|$ converges to $0$ a.s.

Let $\Xi'$ be the probability $1$ set such that for all $
\theta\in\Xi'$, $\bz_{k+1}(\theta) - \bz_{k+1/2}(\theta)\to0$,
$\bz_{k+1/2}(\theta) - \bz_k(\theta) \to 0$, and $\bz_{k+1/2}(\theta)-\bw_k(\theta)
\to 0$. Pick $\theta\in\Xi\cap\Xi'$ and let $\tilde{\mathbf{z}}(\th)$  be a cluster
point of the bounded sequence $(\bz_k(\theta))$. From $\bz_{k+1/2}(\theta) - \bz_k(\theta) \to 0$ and
$ \bz_{k+1/2}(\theta)-\bw_k(\theta)\to 0$ it follows that $\tilde{\mathbf{z}}(\theta)$ is
also a cluster point of $(\bw_k(\theta))$.

By prox-inequality \eqref{eq: prox_ineq} applied to the definition of $\bz_{k+1}$,
\begin{multline}
\langle \bz_{k+1}(\theta) - \bzb_k(\theta) + \tau F(\bw_k(\theta)) - \tau F_{\xi_k}(\bz_{k+1/2}(\theta)) + \tau F_{\xi_k}(\bw_k(\theta)), \bz - \bz_{k+1}(\theta) \rangle \\
+ \tau g(\bz) - \tau g(\bz_{k+1}(\theta)) \geq 0, \quad\forall \bz\in\mathcal{Z}.\label{eq: fix_pt_vi}
\end{multline}
By extracting the subsequence of $\bz_{k}(\theta)$ if needed, taking the
limit along that subsequence and using the lower semicontinuity of
$g$, we deduce that  $\tilde{\mathbf{z}}(\theta)\in\mathsf{Sol}$.
In doing so, we also used that $(\bz_{k+1}(\theta))$ is bounded and
$F_{\xi}$ is continuous for all $\xi$ to deduce
 $\tau\langle F_{\xi_k}(\bw_{k}(\theta)) - F_{\xi_k}(\bz_{k+1/2}(\theta)), \bz - \bz_{k+1}(\theta) \rangle \to 0$. Moreover, since $\bz_{k+1}(\theta) - \bz_k(\theta) \to 0$ and $\bz_{k+1}(\theta)-\bw_k(\theta)\to0$, it follows that $\bz_{k+1}(\theta)-\bzb_{k}(\theta)\to0$.

Hence, all
cluster points of $(\bz_k(\theta))$ and $(\bw_k(\theta))$ belong to
$\mathsf{Sol}$.
We have shown that at least on one subsequence $\alpha\|\bz_k(\theta)
- \tilde{\mathbf{z}}(\th)\|^2 + \frac{1-\alpha}{p}\|\bw_k(\theta) -
\tilde{\mathbf{z}}(\th)\|^2$ converges to $0$. Then, by~\eqref{eq:
  globalization_eq}
we deduce $\alpha\|\bz_k(\theta) - \tilde{\mathbf{z}}(\th)\|^2 + \frac{1-\alpha}{p}\|\bw_k(\theta) - \tilde{\mathbf{z}}(\th)\|^2 \to 0$ and consequently $\|\bz_k(\theta) - \tilde{\mathbf{z}}(\th)\|^2\to 0$. This shows $(\bz_k)$ converges almost surely to a point in $\mathsf{Sol}$.
\end{proof}
\begin{proof}[Proof of~\Cref{lem:prox_update_br}]
  By optimality of $\bz^+$,
\[  0 \in \partial g(\bz^+) + \bu + \a\left(\nabla h(\bz^+)-\nabla
  h(\bz_1)\right) + (1-\a)\left(\nabla h(\bz^+)-\nabla h(\bz_2)\right).\]
This implies by convexity of $g$
\[g(\bz)-g(\bz^+)\geq \lr{\bu + \a\left(\nabla h(\bz^+)-\nabla
    h(\bz_1)\right) + (1-\a)\left(\nabla h(\bz^+)-\nabla h(\bz_2)\right), \bz^+-\bz}.\]
By applying three point identity twice, we deduce
\begin{align*}
g(\bz)-g(\bz^+) + \lr{\bu, \bz-\bz^+}\geq  \a &\left(
  D(\bz, \bz^+) + D(\bz^+, \bz_1) - D(\bz, \bz_1)  \right) \\
+ (1-\a) &\left(
  D(\bz, \bz^+) + D(\bz^+, \bz_2) - D(\bz, \bz_2)  \right)
\end{align*}
and by a simple rearrangement we obtain the result.
\end{proof}

\begin{proof}[Proof of~\Cref{lem: exp_max_lem}]
First, we define the sequence ${\bx}_{k+1} = {\bx}_k + \bu_{k+1}$. It is easy to see that $\bx_{k}$ is $\mathcal{F}_k$-measurable.
Next, by using the definition of $(\bx_{k})$, we have
\begin{equation*}
\| \bx_{k+1} - \bx \|^2 = \|  \bx_k - \bx\|^2 + 2\langle \bu_{k+1},  \bx_k - \bx \rangle + \| \bu_{k+1} \|^2.
\end{equation*}
Summing over $k=0,\dots, K-1$, we obtain
\begin{align*}
\sum_{k=0}^{K-1} 2\langle \bu_{k+1}, \bx -  \bx_k\rangle \leq \| \bx_0 - \bx\|^2 + \sum_{k=0}^{K-1} \| \bu_{k+1} \|^2.
\end{align*}
Next, we take maximum of both sides and then expectation
\begin{equation*}
\mathbb{E}\left[\max_{\bx\in\mathcal{C}}\sum_{k=0}^{K-1} \langle \bu_{k}, \bx\rangle\right] \leq \max_{\bx\in\mathcal{C}}\frac{1}{2}\| \bx_0 - \bx\|^2 +\frac 12 \sum_{k=0}^{K-1} \mathbb{E}\left[\| \bu_{k+1} \|^2\right] + \sum_{k=0}^{K-1} \mathbb{E}\left[\langle \bu_{k+1},  \bx_k \rangle\right].
\end{equation*}
We use the tower property, $\mathcal{F}_k$-measurability of $ \bx_k$, and $\mathbb{E}\left[\bu_{k+1} | \mathcal{F}_k\right] = 0$ to finish the proof, since\\ $\sum_{k=0}^{K-1} \mathbb{E}\left[\langle \bu_{k+1},  \bx_k \rangle\right]=\sum_{k=0}^{K-1} \mathbb{E}\left[ \langle \mathbb{E}\left[\bu_{k+1} | \mathcal{F}_k\right],  \bx_k \rangle \right] = 0$.
\end{proof}

\begin{proof}[Proof of~\Cref{lem: exp_max_lem_br}]
Define for each $s\geq 0$ and for $k\in\{0,\dots, K-1\}$,
\begin{equation*}
 \bx_{k+1}^s = \argmin_{\bx\in\dom g} \{ \langle -\bu_{k+1}^s, \bx \rangle + D(\bx, \bx_k^s) \}, \text{ and let }  \bx_{0}^{s+1} =  \bx_m^{s}.
\end{equation*}
First, we observe $\bx_{k}^s$ is $\mathcal{F}_k^s$-measurable.
By the definition of $ \bx_{k+1}^s$, we have for all $\bx\in \dom g$,
\begin{align*}
\langle \nabla h( \bx_{k+1}^s) - \nabla h( \bx_k^s) - \bu_{k+1}^s, \bx -  \bx_{k+1}^s \rangle \geq 0.
\end{align*}
We apply three point identity to obtain
\begin{align*}
D(\bx,  \bx_k^s) - D(\bx,  \bx_{k+1}^s) - D( \bx_{k+1}^s, \bx_k^s) - \langle \bu_{k+1}^s, \bx -  \bx_{k+1}^s \rangle \geq 0.
\end{align*}
We manipulate the inner product by using H\"older's, Young's inequalities, and strong convexity of $h$,
\begin{align*}
\langle \bu_{k+1}^s, \bx -  \bx_{k+1}^s \rangle &= \langle \bu_{k+1}^s, \bx -  \bx_{k}^s \rangle + \langle \bu_{k+1}^s,  \bx_k^s -  \bx_{k+1}^s \rangle\\
&\geq \langle \bu_{k+1}^s, \bx -  \bx_{k}^s \rangle - \frac{1}{2} \| \bu_{k+1}^s \|^2_{\ast} - \frac{1}{2} \|  \bx_{k+1}^s -  \bx_k^s \|^2\\
&\geq \langle \bu_{k+1}^s, \bx -  \bx_{k}^s\rangle - \frac{1}{2} \| \bu_{k+1}^s \|^2_{\ast} - D( \bx_{k+1}^s,  \bx_k^s),
\end{align*}
which, combined with the previous inequality gives
\begin{align*}
\langle \bu_{k+1}^s, \bx \rangle \leq D(\bx,  \bx_k^s) - D(\bx,  \bx_{k+1}^s) + \langle \bu_{k+1}^s,  \bx_k^s\rangle + \frac{1}{2} \| \bu_{k+1}^s \|^2_\ast.
\end{align*}
We sum this inequality over $k=0,\dots,K-1$ and $s= 0,\dots, S-1$, take
maximum, use ${\bx}_0^{s+1}={\bx}_K^s$ and  the same
derivations as at the end of the proof of~\Cref{lem: exp_max_lem} to show $\sum_{s=0}^{S-1}\sum_{k=0}^{K-1}\mathbb{E}\left[ \lr{\bu_{k+1}^s, \bx_k^s }\right]=0$.
\end{proof}

\section*{Acknowledgments}
Part of the work was done while A. Alacaoglu was at EPFL.
The work of A. Alacaoglu has received funding from the NSF Award  2023239;  DOE  ASCR  under  Subcontract  8F-30039  from  Argonne  National  Laboratory; the European
Research Council (ERC) under the European Union's Horizon 2020
research and innovation programme (grant agreement no $725594$ -
time-data).
The work of Y. Malitsky was supported by the Wallenberg
Al, Autonomous Systems and Software Program (WASP) funded by the Knut
and Alice Wallenberg Foundation.  The project number is 305286.

\printbibliography

\end{document}